\documentclass{amsart}
\usepackage{amsmath,amssymb}
\usepackage{graphicx,subfigure}
\usepackage{url}
\usepackage{color}

\newtheorem{theorem}{Theorem}[section]

\newtheorem{corollary}[theorem]{Corollary}

\newtheorem{lemma}[theorem]{Lemma}

\theoremstyle{definition}
\newtheorem{definition}[theorem]{Definition}

\newtheorem{example}[theorem]{Example}

\theoremstyle{remark}
\newtheorem{remark}[theorem]{Remark}

\numberwithin{equation}{section}

\newcommand{\ep}{\varepsilon}

\DeclareMathOperator{\id}{id}

\newcommand{\supp}{\operatorname{supp}}

\DeclareMathOperator\Vol{Vol} 

\DeclareMathOperator\dist{dist}

\DeclareMathOperator\Exp{Exp}

\DeclareMathOperator\E{\mathbb{E}}

\begin{document}
	
	\title[Equivalence of formulations of Berry's RWM]{Equivalence of local weak and Benjamini--Schramm formulations of Berry's random wave model}
	
\author{Alba Garc\'\i a-Ruiz}
\address{Departamento de Matem\'aticas, CUNEF Universidad, 28040 Madrid, Spain}
\email{alba.garciaruiz@cunef.edu}

	\date{\today}

\begin{abstract}
 Berry's random wave model predicts that high-energy eigenfunctions
of classically chaotic systems can be modeled locally, at the
wavelength scale, by the Gaussian monochromatic random wave.
Several rigorous formulations of this prediction have been proposed.
We prove the equivalence of two such formulations on a closed
manifold: local weak limits and Benjamini--Schramm convergence
of manifolds decorated by eigenfunctions. The latter considers pointed manifolds equipped with
a smooth function, up to isometry, without recording
local orthonormal coordinates. The key step is to recover the unaveraged local function law from its orthogonal average using translation invariance and the ergodicity of the
random wave. Finally, either formulation implies strong inverse localization, allowing high-energy
eigenfunctions to approximate arbitrary solutions of the Euclidean
Helmholtz equation, with applications to their nodal sets.
\end{abstract}

	\maketitle
\section{Introduction}

In his influential papers \cite{Berry1983,Berry1977RegularAI}, Berry proposed a random wave model for high-energy eigenfunctions in systems with chaotic classical dynamics. The precise dynamical assumptions (ergodicity, mixing, chaos, and related properties) under which a rigorous version of this model should hold depend on the formulation considered. In particular, on smooth Riemannian manifolds $(M,g)$ with ergodic geodesic flow, high-frequency Laplace eigenfunctions, that is, functions $\psi_n$ satisfying
\[
\Delta_g\psi_n+\lambda_n\psi_n=0,\qquad\lambda_n\rightarrow\infty
\]
should, after localization around a point at the wavelength scale, behave in a suitable sense like an isotropic Gaussian field $\Psi_{\mathrm{RWM}}$ with covariance
\begin{equation}
	\E\left[\Psi_{\mathrm{RWM}}(x)\Psi_{\mathrm{RWM}}(y)\right]
	=
	\int_{\mathbb{S}^{d-1}}e^{i\langle x-y,\theta\rangle}\,d\omega_{d-1}(\theta),
\end{equation}
where $d\omega_{d-1}$ is the uniform probability measure
on $\mathbb S^{d-1}$. This field is a smooth random solution
of the Helmholtz equation $\Delta\varphi+\varphi=0$
on $\mathbb R^d$, and we denote its law on
$C^\infty(\mathbb R^d)$ by $\mu_{\mathrm{RWM}}$.

This comparison between deterministic eigenfunctions and a stochastic field is known as the Random Wave Model (RWM). It was first introduced in the study of semiclassical quantum systems with chaotic classical dynamics, and an analogous Gaussian local-limit phenomenon has been proved for random regular graphs \cite[Theorem~2]{graph}. The RWM has led to many conjectures concerning $L^p$ norms, semiclassical measures, and the volume and topology of nodal domains of chaotic eigenfunctions. Several of them have been investigated numerically \cite{AURICH1993185,PhysRevE.57.5425,Barnett2005AsymptoticRO,Hejhal1992} or experimentally \cite{PhysRevE.72.066212,HohStock,PhysRevE.70.056209}.

There is, however, no unique rigorous formulation of the local random-wave behavior predicted by Berry's model. The statement that a sequence of deterministic objects has a random limit admits different interpretations, as does the topology in which convergence should take place; see, for instance, \cite{Humphries2017EquidistributionIS,Nonnenmacher2013,zelditch2010recent}. Some formulations concern only the value distribution of eigenfunctions, viewing $\left(M,\frac{dV_g}{\Vol_g(M)}\right)$ as a probability space and each eigenfunction as a random variable. Such formulations do not describe the nodal set, the number of nodal domains, or related geometric features; see also \cite[Introduction, pp.~1--2]{ABLM}. Other approaches retain the local shape of the eigenfunction. Bourgain \cite{Bourgain} used de-randomization to study the nodal domains of toral eigenfunctions through the random wave model. Related methods have also been applied to nodal sets of monochromatic waves satisfying the RWM; see, for instance, \cite{ROMANIEGA20221}.

Two rigorous formulations of this local behavior appear in \cite{ABLM} and \cite{Ingremeau2021}. Both describe the eigenfunction near a randomly chosen point at the wavelength scale, but they encode this localization in different ways. In the local weak limit formulation proposed by Ingremeau, one works in local coordinates (or, in the Riemannian formulation used here, in a chosen local
orthonormal frame $e_p^m$ on $U_m$) and considers rescaled functions of the form
\[
\psi_{p,m}^n(y)
=\psi_n\left(\exp_p\left(
\frac{e_p^m(y)}{\sqrt{\lambda_n}}\right)\right).
\]
The Benjamini--Schramm (BS for short) formulation instead rescales the metric and considers the decorated pointed manifold
\[
[(M,\lambda_n g),p,\psi_n]
\]
up to pointed isometries. Consequently, while the local weak limit formulation records a function on $\mathbb{R}^d$ in chosen coordinates, a Euclidean decorated limit in the Benjamini--Schramm formulation records it only up to an orthogonal change of variables. This difference prevents a direct identification of the limiting measures.

The main purpose of this paper is to show that, on a closed manifold, these two rigorous formulations of the local random-wave prediction are equivalent. The question considered here is therefore the equivalence between two mathematical descriptions of convergence to the random wave model, rather than the dynamical assumptions under which such convergence is expected to occur.

More precisely, by a local frame we mean a smooth family of linear isometries
\[
e_p:\mathbb{R}^d\to T_pM.
\]
For a complete pointed manifold endowed with a smooth real-valued function, we write $[N,p,\phi]$ for its class under isometries preserving both the distinguished point and the function. The smooth topology on these classes is recalled in Section~\ref{sec3}. Our main result is the following.

\begin{theorem}\label{main}
	Let $(M,g)$ be a closed connected smooth Riemannian manifold of dimension
	$d\geq2$. Let $(\psi_n)$ be a real Laplace eigenbasis, normalized by
	\[
	-\Delta_g\psi_n=\lambda_n\psi_n,\qquad
\int_M|\psi_n|^2\,dV_g=\Vol_g(M),
	\qquad \lambda_n\longrightarrow\infty.
	\]
	The following assertions are equivalent:
	\begin{enumerate}
		\item \emph{(Local weak limit formulation)} For some finite family of non-empty pairwise disjoint open sets $U_m$ covering $M$ up to a set of volume zero, with smooth local frames $e_p^m$, the laws of
		\[
		\psi_{p,m}^n(x)
		=
		\psi_n\bigl(\exp_p(\lambda_n^{-1/2}e_p^m x)\bigr),
		\qquad
		p\text{ uniformly distributed in }U_m,
		\]
		converge to $\mu_{\mathrm{RWM}}$ in $C^\infty(\mathbb R^d)$ for every $m$.
		
		\item \emph{(Benjamini--Schramm formulation)} The laws of
		$[(M,\lambda_ng),p,\psi_n]$, with $p$ distributed according to normalized Riemannian volume, converge in the smooth decorated topology to the law of
		$[\mathbb R^d,0,\Psi_{\mathrm{RWM}}]$.
	\end{enumerate}
\end{theorem}
Since assertion~\textup{(ii)} is intrinsic, the equivalence
yields the following consequence.
\begin{corollary}
	If the assertions in Theorem~\ref{main} hold, assertion~\textup{(i)} holds for every such choice of the open sets and the smooth local frames.
\end{corollary}

The eigenbasis is understood to include the constant mode, which is omitted when considering the high-frequency sequence. No assumption on the geodesic flow is required for the equivalence in Theorem~\ref{main} that compares two precise properties of the chosen eigenfunction sequence.

Local weak convergence and Benjamini--Schramm convergence are closely related for graphs. For finite graphs with a uniform bound on vertex degrees, Benjamini--Schramm convergence is local weak convergence
after choosing the root uniformly at random. Equivalently,
for every fixed radius, the distribution of the rooted ball
converges to that of the corresponding ball in a random
rooted limit graph; see \cite{AldousLyons2007,AldousSteele2004, benjamini2001recurrence}. The situation here is different: the LWL formulation (Definition~\ref{C1}) depends on a choice of local orthonormal frames, whereas the BS formulation (Definition~\ref{C2}) is intrinsic and is defined on $\mathcal{E}^d$ modulo pointed isometries. Bridging this difference requires control in rescaled exponential coordinates and the recovery of a function law from its orthogonal average. The proof uses tightness of certain measures introduced in Section~\ref{sec3} together with the ergodicity of $\mu_{\mathrm{RWM}}$. This equivalence is independent of the separate question of whether Berry's RWM holds on a given chaotic system.

In the final section, we introduce the inverse localization property and show that Berry's property, in either of the equivalent formulations above, implies inverse localization. Roughly speaking, a compact manifold $M$ satisfies this property if arbitrary solutions of the Helmholtz equation on $\mathbb{R}^d$ can be approximated by high-energy eigenfunctions of the Laplace--Beltrami operator on $M$. All currently known examples rely on high spectral degeneracy in the construction of the approximating eigenfunction. The version obtained from Berry's property is stronger: there exists a sequence of eigenfunctions (associated with eigenvalues along the prescribed sequence) such that any Helmholtz solution can be approximated by members of this same sequence, localized around suitable points of $M$. Thus the sequence is independent of the target solution, and the freedom lies in the observation point rather than in choosing a new eigenfunction from an eigenspace. Consequently, no assumption on spectral multiplicities is needed
for this implication.

\subsection*{Organization of the paper}
In Section~\ref{sec2}, we introduce the geometric setting,
	the random monochromatic wave, and the notation used throughout
	the paper. In Section~\ref{sec3}, we state the local weak limit
	and Benjamini--Schramm formulations of the Berry property
	and establish some properties of the local laws.
	In Section~\ref{sec:framed}, we introduce the framed formulation,
	prove its equivalence with the Benjamini--Schramm formulation,
	and relate it to convergence of local functions.
	Section~\ref{sec:proof-main} proves the two implications of
	Theorem~\ref{main}. Finally, in Section~\ref{sec6}, we introduce
	the inverse localization property, prove that the Berry property
	implies its strong form, and conclude with some consequences
	and remarks.
	\section{Preliminaries}\label{sec2}
\subsection{Definitions on the manifold}
Throughout the paper, $(M,g)$ denotes a compact connected $C^\infty$ Riemannian manifold of dimension $d\geq2$ without boundary. We use $\mathbb N=\{1,2,\ldots\}$ and
$\mathbb N_0=\mathbb N\cup\{0\}$.  We write $r_{\mathrm{inj}}:=\operatorname{inj}(M,g)$, $dV_g$ for its Riemannian volume measure, $\Vol_g(M)$ for its volume, and $\Delta_g$ for the Laplace--Beltrami operator. By the spectral theorem, there exists an orthogonal real eigenbasis $(\psi_n)_{n\in\mathbb N_0}$ of $L^2(M,dV_g)$ such that
\[
\Delta_g\psi_n+\lambda_n\psi_n=0,
\qquad
\|\psi_n\|_{L^2(M,dV_g)}^2=\Vol_g(M),
\]
with $0=\lambda_0<\lambda_1\leq\lambda_2\leq\cdots\to\infty$.
The constant mode is omitted in all rescaled expressions below,
which use $n\geq1$. The frequency corresponding to $\lambda_n$ is
$\sqrt{\lambda_n}$. For $p\in M$, we denote by $\exp_p:T_pM\to M$ the exponential map and by $\dist_g$ the Riemannian distance. All $L^p(M)$ norms are taken with respect to $dV_g$. We write $B_g(p,R)$ for geodesic balls and $B_R:=B(0,R)$
for Euclidean balls centered at the origin. For a Euclidean
ball $B$ and a function defined near $\overline B$, we use
\[
\|f\|_{C^r(\overline B)}
:=\max_{|\alpha|\leq r}\sup_{x\in\overline B}
|\partial^\alpha f(x)|.
\]

We use the following finite decomposition of $M$.

\begin{lemma}\label{frames}
	Let $(M,g)$ be a compact connected Riemannian manifold of dimension $d$ without boundary. Then there exists a finite family of nonempty pairwise disjoint open sets $	\{U_m\}_{m=1}^{m_{\max}}$	such that 	\[
	M\subset\bigcup_{m=1}^{m_{\max}}\overline{U_m};
	\] and, for every $1\leq m\leq m_{\max}$, the following properties hold:
	\begin{enumerate}
		\item
		\[
		\Vol_g(\partial U_m)=0;
		\]
		\item the tangent bundle $TU_m$ admits a smooth orthonormal frame $(V_1^m,\ldots,V_d^m).$
	\end{enumerate}
\end{lemma}

\begin{proof}
	Fix $0<\ep<r_{\mathrm{inj}}$, where $r_{\mathrm{inj}}$ is the
	injectivity radius of $M$, and let
	$\mathbf E=\{p_1,\ldots,p_{m_{\max}}\}\subset M$ be a maximal
	$\ep$-separated set. Since $M$ is compact, $\mathbf E$ is finite.
	Moreover, by maximality, every point has a nearest center $p_m$ at distance
	at most $\ep<r_{\mathrm{inj}}$. 
	
	For each $p_m\in\mathbf E$, define the open Voronoi cell
	\begin{equation}
		U_m
		:=
		\left\{
		p\in M:
		\dist_g(p,p_m)<\dist_g(p,p_k)
		\text{ for all }k\neq m
		\right\}.
	\end{equation}
	The sets $U_m$ are open and pairwise disjoint. Moreover,
	\[
	U_m\subset B_g(p_m,\ep),
	\]
	since otherwise a point of $U_m$ lying at distance at least $\ep$
	from $p_m$ would be at distance at least $\ep$ from every point
	of $\mathbf E$, contradicting maximality.
	
	For $m\neq k$, the function
	\[
	q\longmapsto \dist_g(q,p_m)^2-\dist_g(q,p_k)^2
	\]
	is smooth wherever both distances are less than
	$r_{\mathrm{inj}}$, with gradient
	\[
	2\bigl(\exp_q^{-1}(p_k)-\exp_q^{-1}(p_m)\bigr)\neq0.
	\]
The set where the two distances coincide (zero of the function) is therefore
a smooth hypersurface in this region.
	The finitely many such zero sets cover
	$M\setminus\bigcup_m U_m$, which contains every
	$\partial U_m$. Hence this complement and all the
	boundaries have volume zero. Thus $\bigcup_m U_m$ is
	dense, and finiteness gives
	\[
	M=\overline{\bigcup_m U_m}=\bigcup_m\overline{U_m}.
	\]
	
	Finally, since $U_m\subset B_g(p_m,\ep)$ and
	$\ep<r_{\mathrm{inj}}$, the ball $B_g(p_m,\ep)$ is diffeomorphic
	to a Euclidean ball and admits a smooth orthonormal frame.
	Restricting it to $U_m$ proves the last assertion.
	\end{proof}

	Fix on each $U_m$ the smooth orthonormal frame $(V_1^m,\ldots,V_d^m),$
	introduced above. For $p\in U_m$, define
	\begin{equation}\label{trivialframe}
		\begin{aligned}
			e_p^m:\mathbb{R}^d&\longrightarrow T_pM,\\
			y&\longmapsto \sum_{j=1}^d y_jV_j^m(p).
		\end{aligned}
	\end{equation}
	We also write \begin{equation}
		 w_m:=\frac{\Vol_g(U_m)}{\Vol_g(M)}.
	\end{equation}
Since the sets $U_m$ are pairwise disjoint and cover $M$ up to
their boundaries, which have zero volume, we have
$w_m>0$ and
\begin{equation}
	\sum_{m=1}^{m_{\max}}w_m=1.
\end{equation}
\subsection{Smooth functions}\label{sec2.2}

We equip $C^\infty(M)$ and $C^\infty(\mathbb{R}^d)$ with the topology of uniform convergence of all derivatives on compact sets. For later use, we describe explicitly a metric inducing this topology on $C^\infty(\mathbb{R}^d)$. For $k\in\mathbb{N}_0$ and $L\in\mathbb{N}$, set
\begin{equation}
	|f|_{k,L}
	:=
	\max_{|\alpha|\leq k}
	\sup_{x\in \overline{B_L}}
	|\partial^\alpha f(x)|,
\end{equation}
and define
\begin{equation}\label{dist}
	\dist(f,g)
	=
	\sum_{L=1}^\infty\sum_{k=0}^\infty
	2^{-L-k}
	\frac{|f-g|_{k,L}}{1+|f-g|_{k,L}}.
\end{equation}
Then the space $C^\infty(\mathbb R^d)$ is a Fréchet space
\cite[Chapter~10, Example~I]{Treves1967}, and polynomials are dense in it \cite[Chapter~15, Corollary~4 to Theorem~15.3]{Treves1967}. Since
polynomials with coefficients in $\mathbb Q$ form a countable
set that is dense in the space of all polynomials, $C^\infty(\mathbb R^d)$ is separable. Being completely metrizable, it is therefore a Polish space.

A set $A\subset C^\infty(\mathbb{R}^d)$ is bounded if every seminorm $|\cdot|_{k,L}$ is bounded on $A$. This is stronger than boundedness with respect to the metric $\dist$, since a $\dist$-ball does not in general control all the seminorms. By the Arzelà--Ascoli theorem and a diagonal argument, bounded subsets of $C^\infty(\mathbb{R}^d)$ are relatively compact; hence closed bounded subsets are compact.

All probability measures considered below are understood with respect to their Borel $\sigma$-algebras of the corresponding topological spaces. For a topological space $E$, we write $\mathcal{B}(E)$ for its Borel $\sigma$-algebra and $C_b(E)$ for the space of bounded continuous real-valued functions on $E$. For a Borel probability measure $\mu$ on $E$, we write $\mu(F):=\int_E F\,d\mu$ for $F\in C_b(E)$.

We repeatedly use the following notion of convergence on the different Polish spaces appearing in the paper, in particular $C^\infty(\mathbb R^d)$, $C^\infty(M)$, $\mathcal M^d$, and $\mathcal E^d$, that are introduced in Section \ref{sec3}.

\begin{definition}[Weak convergence of probability measures]
	Let $E$ be a Polish space and let $\mu_n,\mu$ be Borel probability measures on $E$. We write
	\[
	\mu_n\Rightarrow\mu
	\]
when
	\[
	\int_E F\,d\mu_n\longrightarrow\int_E F\,d\mu
	\qquad\text{for every }F\in C_b(E).
	\]
	We refer to this as weak convergence. Some authors refer to this convergence as the weak$^*$ convergence.
\end{definition}

\subsection{The random monochromatic wave}\label{sec:rwm}

A smooth random field on $\mathbb R^d$ is a measurable map
\[
X:(\Omega,\mathcal B,P)\longrightarrow C^\infty(\mathbb R^d).
\]
Its law is the Borel probability measure $\mu_X:=P\circ X^{-1}$.
The field is Gaussian if every finite vector
$(X(t_1),\ldots,X(t_k))$ has a multivariate Gaussian distribution.
Its law is then determined by its mean
$m(t):=\mathbb E[X(t)]$ and covariance
$K(t,s):=\operatorname{Cov}(X(t),X(s))$.
Moreover, it is centered if $m\equiv0$, stationary if $m$ is constant
and $K(t,s)=C(t-s)$, and isotropic if, in addition,
$C(t-s)$ depends only on $|t-s|$.

We denote by $\Psi_{\mathrm{RWM}}$ the centered stationary
Gaussian field on $\mathbb R^d$ with covariance
\begin{equation}\label{covRWM}
	\mathbb{E}\left[
	\Psi_{\mathrm{RWM}}(x)\Psi_{\mathrm{RWM}}(y)
	\right]
	=
	\int_{\mathbb{S}^{d-1}}
	e^{i\langle x-y,\theta\rangle}\,
	d\omega_{d-1}(\theta),
\end{equation}
where $d\omega_{d-1}$ is the uniform probability measure
on $\mathbb S^{d-1}$.

The measure whose Fourier transform gives the stationary
covariance is called the spectral measure; see
\cite[Theorem~2.1.3]{cohen2013fractional}.
Thus the spectral measure of $\Psi_{\mathrm{RWM}}$
is the uniform probability measure on the unit sphere.
Its compact support ensures that the field admits a smooth
version \cite[Sections~A.10--A.11]{NZ2016}, whose law on
$C^\infty(\mathbb R^d)$ we denote by $\mu_{\mathrm{RWM}}$.

Since its spectral measure is supported on $\mathbb S^{d-1}$,
the field satisfies
\[
-\Delta\Psi_{\mathrm{RWM}}=\Psi_{\mathrm{RWM}}
\]
almost surely. Hence $\mu_{\mathrm{RWM}}$ is supported on
the closed space of monochromatic Helmholtz solutions
\[
\mathrm{FP}
:=
\left\{
f\in C^\infty(\mathbb R^d):
-\Delta f=f
\right\},
\]
which we equip with the topology inherited from
$C^\infty(\mathbb R^d)$.
We may therefore also regard $\mu_{\mathrm{RWM}}$
as a probability measure on $\mathrm{FP}$.

Moreover, $\mu_{\mathrm{RWM}}$ has full support in
$\mathrm{FP}$: every nonempty open subset of this space
has positive measure. This support property is justified
in Section~\ref{sec6}, where it is used to deduce
strong inverse localization.

The covariance in \eqref{covRWM} is invariant under simultaneous
translations and orthogonal transformations of its arguments.
Since a centered Gaussian law is determined by its covariance,
$\mu_{\mathrm{RWM}}$ is invariant under translations
and orthogonal changes of coordinates. In particular, the maps
\[
(\tau_xf)(y):=f(x+y),\qquad x,y\in\mathbb R^d,
\]
define a measure-preserving action of $\mathbb R^d$ on
$(C^\infty(\mathbb R^d),\mathcal B(C^\infty(\mathbb R^d)),
\mu_{\mathrm{RWM}})$.
For $d\geq2$, the spectral measure has no atoms, so the
Fomin--Grenander--Maruyama theorem
\cite[Appendix~B]{NZ2016} implies that this action is ergodic.
This means that every Borel set $A$ satisfying
\[
\mu_{\mathrm{RWM}}\bigl((\tau_xA)\triangle A\bigr)=0
\qquad\text{for every }x\in\mathbb R^d
\]
has measure zero or one, where $\triangle$ denotes
symmetric difference.

Ergodicity also makes $\mu_{\mathrm{RWM}}$ an extreme point
of the convex set of translation-invariant probability measures.
Thus, if
\[
\mu_{\mathrm{RWM}}=t\nu_1+(1-t)\nu_2,\qquad 0<t<1,
\]
with $\nu_1$ and $\nu_2$ translation invariant, then
$\nu_1=\nu_2=\mu_{\mathrm{RWM}}$.
Indeed, their Radon--Nikodym derivatives with respect to
$\mu_{\mathrm{RWM}}$ are translation invariant and hence
constant almost surely by ergodicity; normalization makes
both constants equal to one. This is the argument of
\cite[Theorem~4.4]{einsiedler2010ergodic}, applied here to
the translation action.
These properties are used in the proof of
Theorem~\ref{main}.

We also use the following integral form of the extremality
of ergodic measures.

\begin{lemma}\label{int}
Let $\mu$ be a translation-ergodic Borel probability measure on
$C^\infty(\mathbb R^d)$. Let $(Y,\mathcal B_Y,m)$ be a
	probability space and let $y\mapsto v_y$ be a measurable family
	of translation-invariant Borel probability measures on
	$C^\infty\left(\mathbb{R}^d\right)$. If
	\begin{equation}
		\mu=\int_Y v_y\,dm(y),
	\end{equation}
	then $v_y=\mu$ for $m$-almost every $y\in Y$.
\end{lemma}
Here measurability means that $y\mapsto v_y(A)$ is
measurable for every Borel set
$A\subset C^\infty(\mathbb R^d)$.
\begin{proof}
Since $C^\infty(\mathbb R^d)$ is Polish, there exists a
countable family
$\mathcal D\subset C_b(C^\infty(\mathbb R^d))$
that determines Borel probability measures; see
\cite[Chapter~3, proof of Proposition~4.4, p.~112]{EthierKurtz1986}. If $v_y\neq\mu$ on a
	set of positive $m$-measure, then, by countability of $\mathcal{D}$,
	there exists $F\in\mathcal{D}$ such that one of the sets
	\begin{equation}
		A_F^+:=
		\left\{y\in Y:
		v_y(F)>\mu(F)\right\},
		\qquad
		A_F^-:=
		\left\{y\in Y:
		v_y(F)<\mu(F)\right\}
	\end{equation}
	has positive measure. Without loss of generality, assume that
	$m(A_F^+)>0$. The identity
	$\mu=\int_Yv_y\,dm(y)$ excludes $m(A_F^+)=1$, and therefore
	$0<m(A_F^+)<1$. Define
	\begin{equation}
		v_+:=
		\frac{1}{m(A_F^+)}
		\int_{A_F^+}v_y\,dm(y),
		\qquad
		v_-:=
		\frac{1}{m((A_F^+)^c)}
		\int_{(A_F^+)^c}v_y\,dm(y).
	\end{equation}
	Both $v_+$ and $v_-$ are translation-invariant probability measures
	and
	\begin{equation}
		\mu=m(A_F^+)v_+
		+\left(1-m(A_F^+)\right)v_-.
	\end{equation}
	By the extremality of the ergodic measure $\mu$, this implies
	$v_+=\mu$. On the other hand,
	\begin{equation}
		v_+(F)
		=
		\frac{1}{m(A_F^+)}
		\int_{A_F^+}v_y(F)\,dm(y)
		>
		\mu(F),
	\end{equation}
	a contradiction. Hence $v_y=\mu$ for $m$-almost every $y$.
\end{proof}

	\section{Two formulations of the Berry property}\label{sec3}
	\subsection{Berry's conjecture in the local weak limit form}	We first consider the formulation of the conjecture proposed in \cite{Ingremeau2021} (see also \cite{Ingremeau2022}). We refer the reader to the following lecture notes \cite{ingremeau_lecture_notes} for a more detailed introduction. 

 For \(p\in U_m\), define
\[
\Exp_p^m(y):=\exp_p(e_p^m(y))=\exp_p\left(\sum_{j=1}^d y_jV_j^m(p)\right).
\]
Let $p$ be uniformly distributed on $U_m$ with respect to
$dV_g/\Vol_g(U_m)$. For $n\geq1$, we recall that $\psi_n$ is the $n$th eigenfunction of the Laplace--Beltrami operator on $M$ and we define the random field
\begin{equation}
	\psi_{p,m}^n(y):=
	\psi_n\left(\Exp_p^m\left(\frac{y}{\sqrt{\lambda_n}}\right)\right).
\end{equation}
The smooth dependence on $(p,y)$ implies that
$p\mapsto\psi_{p,m}^n$ is continuous in the
$C^\infty(\mathbb R^d)$ topology, and hence Borel measurable,
so it defines a random element of this space. This definition depends on \(U_m\) and on the chosen frame.

We denote by \(\mu_m^n\) the law of \(\psi_{p,m}^n\), namely
\begin{equation}\label{maxim}
	\mu_m^n:=
	\frac{1}{\Vol_g(U_m)}
	\int_{U_m}\delta_{\psi_{p,m}^n}\,dV_g(p),
\end{equation}
so that, for every bounded continuous functional
\(F:C^\infty(\mathbb R^d)\to\mathbb R\), we have
\[
\mu_m^n(F)=
\frac{1}{\Vol_g(U_m)}
\int_{U_m}F(\psi_{p,m}^n)\,dV_g(p).
\]
\begin{definition}[Convergence in law]
	Let \(\mu_{\mathrm{RWM}}\) be the probability measure on
\(C^\infty(\mathbb R^d)\) induced by \(\Psi_{\mathrm{RWM}}\). We say that \(\left(\psi_n\right)_{n\in\mathbb{N}}\) converges in law to \(\Psi_{\mathrm{RWM}}\) on \(U_m\), with respect to the frame \((V_1^m,\ldots,V_d^m)\), if
	\[
	\mu_m^n\Rightarrow\mu_{\mathrm{RWM}}
	\qquad\text{as }n\to\infty.
	\]
	We say that \(\psi_n\) converges in law to \(\Psi_{\mathrm{RWM}}\) on \(M\) if this convergence holds for every \(m=1,\ldots,m_{\max}\), in a family as in Lemma~\ref{frames}, along the same sequence $n\to\infty$.
\end{definition}
	\begin{remark}
		Two useful examples of functionals in \(C_b(C^\infty(\mathbb{R}^d))\) are the following:
		\begin{itemize}
			\item for any \(\chi\in C_b(\mathbb{R})\), $H(f):=\chi(f(0));$
			\item for any \(h\in C^\infty(\mathbb{R}^d)\), \(r\in\mathbb{N}_0\), compact ball \(\overline B\subset\mathbb{R}^d\), and \(\chi\in C_b([0,\infty))\),
			\[
			G(f):=\chi\!\left(\|f-h\|_{C^r\left(\overline B\right)}\right).
			\]
		\end{itemize}
		Both are bounded and continuous with respect to the \(C^\infty\)-topology.
	\end{remark}
The local weak limit formulation introduced leads to the following property.
\begin{definition}[The Berry property in the local weak limit form]\label{C1}
	We say that the sequence $(\psi_n)_{n\in\mathbb{N}}$ satisfies the Berry property in the local weak limit form if, for some family $\{U_m\}_{m=1}^{m_{\max}}$ as in Lemma~\ref{frames}, and some choice of smooth
	orthonormal frames on these sets, \(\left(\psi_n\right)_{n\in\mathbb{N}}\) converges in law to \(\Psi_{\mathrm{RWM}}\) on \(M\).
\end{definition}
\begin{remark}
Although Definition~\ref{C1} depends \emph{a priori} on the choice of the sets \(U_m\) and the corresponding local orthonormal frames, Theorem~\ref{main} implies that this property is independent of the chosen local frames and sampling partition. An analogous formulation can be introduced for manifolds with boundary when \((\psi_n)_{n\in\mathbb{N}}\) is a sequence of Dirichlet eigenfunctions; see \cite{Ingremeau2021}.
\end{remark}

Berry's model suggests that this property should hold for
suitably chaotic geodesic flows. The precise dynamical
assumptions remain open; we study the property itself
without addressing this question.
	
\subsection{Tightness and translation invariance}
In this section, we state and prove some properties of the previously introduced formulation of the Berry property such as subsequential compactness and translation invariance of all limits.

\begin{lemma}\label{exist}
For any $1\leq m\leq m_{\max}$, there exists a subsequence $\left(n_k\right)_{k\in\mathbb{N}}$ and there exists a probability measure $\mu_m$ on $C^\infty\left(\mathbb{R}^d\right)$ such that we have $ \mu_m^{n_k}\Rightarrow\mu_m$. \end{lemma}
\begin{proof}
	Since $C^\infty(\mathbb{R}^d)$, endowed with the distance introduced in
	\eqref{dist}, is a Polish space, Prokhorov's theorem
	\cite[Section 5]{Billingsley1999} reduces the proof to showing that the
	family $\{\mu_m^n\}_{n\geq1}$ is tight. Thus, for every $\varepsilon>0$,
	we need to find a compact set $K_\varepsilon\subset
	C^\infty(\mathbb{R}^d)$ such that
	\begin{equation}
		\mu_m^n\bigl(C^\infty(\mathbb{R}^d)\setminus K_\varepsilon\bigr)
		<\varepsilon, \qquad 	\text{for all } n\in\mathbb{N}.
	\end{equation}
 This is essentially the content of
	\cite[Lemma 1]{Ingremeau2021}; see also
	\cite[Proposition 8]{ingremeau_lecture_notes}.
	
	Fix $\varepsilon>0$. Since $\Vol_g(\partial U_m)=0$, we choose
	$\delta>0$ such that the compact set
	\begin{equation}
		U_m^\delta:=
		\{p\in U_m:\dist_g(p,\partial U_m)\geq\delta\},
	\end{equation}
satisfies
	\begin{equation}\label{boundarytight}
		\Vol_g(U_m\setminus U_m^\delta)
		<
		\frac{\varepsilon}{2}\Vol_g(U_m).
	\end{equation}
	
	Fix now $s\in\mathbb N$ and $R>0$. Recall that
$H^s(B_R)$ consists of $L^2$ functions whose weak derivatives
up to order $s$ also belong to $L^2$. We claim that there is a constant
$C_{M,R}^{\delta,s}$, independent of $n$, such that
\begin{equation}\label{Hsbound}
	\int_{U_m^\delta}
	\|\psi_{p,m}^n\|_{H^s(B_R)}^2\,dV_g(p)
	\leq C_{M,R}^{\delta,s}.
\end{equation}
Indeed, for $|\alpha|\leq s$, the chain rule applied to
\[
\psi_{p,m}^n(y)
=
\psi_n\left(
\Exp_p^m\left(\frac{y}{\sqrt{\lambda_n}}\right)
\right)
\]
gives, uniformly for $p\in U_m^\delta$ and $y\in B_R$,
\begin{equation}\label{chainbound}
	\left|
	D_y^\alpha\psi_{p,m}^n(y)
	\right|^2
	\leq
	C_{M,R}^{\delta,s}
	\sum_{j=0}^{|\alpha|}
	\lambda_n^{-|\alpha|}
	\left|
	\nabla^j\psi_n\left(
	\Exp_p^m\left(\frac{y}{\sqrt{\lambda_n}}\right)
	\right)
	\right|^2.
\end{equation}
Here $D_y^\alpha$ denotes Euclidean partial differentiation,
and $\nabla^j$ denotes the $j$th covariant derivative on $M$.

	Here $	C_{M,R}^{\delta,s}$ is uniform because the frame and the exponential map,
	together with the finitely many derivatives involved, are uniformly
	bounded on the compact region under consideration.
	
For fixed $y\in B_R$, consider
\begin{equation}
	Q_{n,y}(p):=
	\Exp_p^m\left(\frac{y}{\sqrt{\lambda_n}}\right).
\end{equation}
Uniformly for $y\in B_R$, $	Q_{n,y}\longrightarrow\id$ in $C^1(U_m^\delta)$
as $n\to\infty$. To justify global injectivity, choose
$\chi\in C_c^\infty(U_m)$ equal to one near
$\overline{U_m^\delta}\subset U_m$. The maps
\[
p\longmapsto
\exp_p\left(\frac{\chi(p)e_p^m(y)}{\sqrt{\lambda_n}}\right)
\]
extend by the identity outside $U_m$ and converge to
$\id_M$ in $C^1(M)$, uniformly for $|y|\leq R$.
Since $C^1$ diffeomorphisms of the compact manifold $M$ form
a $C^1$-open subset of $C^1(M,M)$, these smooth extensions
are diffeomorphisms for all sufficiently large $n$,
simultaneously for $|y|\leq R$. They agree with $Q_{n,y}$ on a neighborhood of
$U_m^\delta$, so $Q_{n,y}$ is injective there.
Moreover, the uniform $C^1$ convergence implies that their
Jacobians with respect to Riemannian volume converge
uniformly to one. Thus the Jacobian of $Q_{n,y}$ on
$U_m^\delta$ is bounded above and below by positive
constants independent of $n$ and $y$, for all sufficiently
large $n$ and $|y|\leq R$.

 Denoting
\[
U_{m,n,y}^\delta:=Q_{n,y}(U_m^\delta),
\]
	Fubini's theorem, \eqref{chainbound}, and the change of variables
	$q=Q_{n,y}(p)$ give
	\begin{equation}
		\begin{aligned}
		&	\int_{U_m^\delta}
			\|\psi_{p,m}^n\|_{H^s(B_R)}^2\,dV_g(p)
			\\&\leq
			C_{M,R}^{\delta,s}
			\sum_{|\alpha|\leq s}
			\sum_{j=0}^{|\alpha|}
			\lambda_n^{-|\alpha|}
			\int_{B_R}
			\int_{U_{m,n,y}^\delta}
			|\nabla^j\psi_n(q)|^2\,dV_g(q)\,dy \\
			&\leq
			C_{M,R}^{\delta,s}
			\sum_{|\alpha|\leq s}
			\sum_{j=0}^{|\alpha|}
			\lambda_n^{-|\alpha|}
		 \left|B_R\right|
			\|\nabla^j\psi_n\|_{L^2(M)}^2.
		\end{aligned}
	\end{equation}
On the closed manifold $M$, elliptic regularity gives the
equivalence of the Sobolev norm defined by covariant derivatives
and the spectral norm $\|(I-\Delta_g)^{j/2}u\|_{L^2(M)}$.
Applying this to $-\Delta_g\psi_n=\lambda_n\psi_n$ yields
	\begin{equation}
		\|\nabla^j\psi_n\|_{L^2(M)}^2
		\leq
		C_{M,j}\lambda_n^j\|\psi_n\|_{L^2(M)}^2.
	\end{equation}
Since $j\leq|\alpha|$ and $\lambda_n\geq\lambda_1>0$,
the factors $\lambda_n^{j-|\alpha|}$ are uniformly bounded.
Together with
$\|\psi_n\|_{L^2(M)}^2=\Vol_g(M)$, this proves
\eqref{Hsbound}. The finitely
	many values of $n$ discarded above can be absorbed by enlarging the
	constant.
	
	By Sobolev embedding, for every $l\in\mathbb N_0$ and $R\in\mathbb N$
	we may choose $s$ sufficiently large in \eqref{Hsbound} to obtain
	\begin{equation}\label{usefulineq}
		\int_{U_m^\delta}
		\|\psi_{p,m}^n\|_{C^l(\overline{B_R})}^2\,dV_g(p)
		\leq C_{M,R,\delta,l},
	\end{equation}
	uniformly in $n$.
	
	We now use these estimates simultaneously for all seminorms. For
	$l\in\mathbb N_0$ and $R\in\mathbb N$, choose constants
	$A_{l,R}>0$ sufficiently large so that Markov's inequality and
	\eqref{usefulineq} give
	\begin{equation}\label{markovtight}
		\Vol_g\left\{
		p\in U_m^\delta:
		\|\psi_{p,m}^n\|_{C^l(\overline{B_R})}>A_{l,R}
		\right\}
		<
		\varepsilon\,2^{-l-R-2}\Vol_g(U_m)
	\end{equation}
	for every $n$. Define
	\begin{equation}
		K_\varepsilon:=
		\left\{
		f\in C^\infty(\mathbb{R}^d):
		\|f\|_{C^l(\overline{B_R})}\leq A_{l,R}
		\text{ for all }l\in\mathbb N_0,\ R\in\mathbb N
		\right\}.
	\end{equation}
	This set is closed and bounded in $C^\infty(\mathbb{R}^d)$ and hence
	compact by the compactness criterion recalled in Section~\ref{sec2}.
	Moreover, by \eqref{boundarytight}, \eqref{markovtight}, and the union
	bound,
	\begin{equation}
		\begin{aligned}
			\Vol_g\{p\in U_m:\psi_{p,m}^n\notin K_\varepsilon\}
			&\leq
			\Vol_g(U_m\setminus U_m^\delta)
			+
			\sum_{l=0}^\infty\sum_{R=1}^\infty
			\varepsilon\,2^{-l-R-2}\Vol_g(U_m)\\
			&<
			\varepsilon\,\Vol_g(U_m).
		\end{aligned}
	\end{equation}
	Therefore
	\begin{equation}
		\mu_m^n\bigl(C^\infty(\mathbb{R}^d)\setminus K_\varepsilon\bigr)
		<\varepsilon
	\end{equation}
	uniformly in $n$. Thus $\{\mu_m^n\}_{n\geq1}$ is tight, and Prokhorov's
	theorem yields a subsequence $(n_k)_{k\in\mathbb{N}}$ and a Borel probability measure
	$\mu_m$ on $C^\infty(\mathbb{R}^d)$ such that
	\begin{equation}
		\mu_m^{n_k}\Rightarrow\mu_m.
	\end{equation}
\end{proof}
As can be seen in \cite[Lemma 8]{Ingremeau2021}, for any $1\leq m\leq m_{\max}$, this limit $\mu_m$ is translation invariant. We recall the proof of this fact here, but first we introduce and recall some notation.

\begin{definition}
	Let $y\in\mathbb{R}^d$. For any $f\in C^\infty\left(\mathbb{R}^d\right)$, we define $\tau_y f\in C^\infty\left(\mathbb{R}^d\right)$ by $(\tau_yf)(\cdot)=f(y+\cdot)$. One can check that $\tau_y$ is a homeomorphism of $C^\infty\left(\mathbb{R}^d\right)$.
	
	For any functional $F$ on $C^\infty\left(\mathbb{R}^d\right)$, we define $\tau_yF$ by $(\tau_yF)(f)=F(\tau_yf)$ for all $f\in C^\infty\left(\mathbb{R}^d\right)$. 
	
	Finally, for a Borel probability measure $\mu$ on $C^\infty(\mathbb R^d)$,
	we write $\tau_y\mu:=(\tau_y)_*\mu$. Thus
	$(\tau_y\mu)(F)=\mu(\tau_yF)$ for every
	$F\in C_b(C^\infty(\mathbb R^d))$.
\end{definition}

\begin{lemma}\label{translation}
If $\mu_m^{n_k}\Rightarrow\mu_m$ in $C^\infty\left(\mathbb{R}^d\right)$, then $\mu_m$ is
invariant under translations.
\end{lemma}

\begin{proof}
	Fix $y\in\mathbb{R}^d$ and $F\in C_b(C^\infty(\mathbb{R}^d))$. It is enough to prove that
	\begin{equation}\label{goal}
		\mu_m^n(\tau_yF)-\mu_m^n(F)\longrightarrow0.
	\end{equation}
	Indeed, taking a convergent subsequence $\left(n_k\right)_{k\in\mathbb{N}}$ and using $\mu_m^{n_k}\Rightarrow\mu_m$ then gives
	$\mu_m(\tau_yF)=\mu_m(F)$.
	
	We first work away from the boundary of $U_m$. To do so, let
	$a\in C_c^\infty(U_m)$, $0\leq a\leq1$, and set
	$P:=\supp a\subset U_m$. Consider
	\begin{equation}
		Q_{n,y}(p)=
		\Exp_p^m\left(\frac{y}{\sqrt{\lambda_n}}\right).
	\end{equation}
	For $p\in P$ and $x$ in a fixed ball $B_L\subset\mathbb{R}^d$, both
	$Q_{n,y}(p)$ and
	$\Exp_{Q_{n,y}(p)}^m(x/\sqrt{\lambda_n})$ lie in a neighborhood
	of $p$ for all sufficiently large $n$. We may therefore define
	\begin{equation}\label{An}
		A_n(p,x):=
		\sqrt{\lambda_n}\,
		(\Exp_p^m)^{-1}
		\left(
		\Exp_{Q_{n,y}(p)}^m
		\left(\frac{x}{\sqrt{\lambda_n}}\right)
		\right).
	\end{equation}
The inverse of $\Exp_p^m$ is taken on the normal ball
$B_g(p,r_{\mathrm{inj}})$.	By construction,
	\begin{equation}\label{shift}
		\psi_{Q_{n,y}(p),m}^n(x)
		=
		\psi_{p,m}^n\bigl(A_n(p,x)\bigr).
	\end{equation}
	
	Next we see that the map $A_n$ is asymptotic to the translation $x\mapsto x+y$,
	uniformly for $p\in P$. More precisely, for every $r\in\mathbb N_0$ and
	$L>0$ there exists $C_{P,L,r,y}>0$ such that
	\begin{equation}\label{Aest}
		\sup_{p\in P}
		\left\|
		A_n(p,\cdot)-(\id+y)
		\right\|_{C^r(\overline{B_L})}
		\leq
		C_{P,L,r,y}\lambda_n^{-1/2}.
	\end{equation}
	To see this, set
	\begin{equation}
		\Phi(t,p,x):=
		(\Exp_p^m)^{-1}
		\left(
		\Exp_{\Exp_p^m(ty)}^m(tx)
		\right).
	\end{equation}
	For $|t|$ small this map is smooth on a neighborhood of
	$\{0\}\times P\times\overline{B_L}$, and
	\begin{equation}
		\Phi(0,p,x)=0,
		\qquad
		\partial_t\Phi(0,p,x)=x+y.
	\end{equation}

Indeed, moving the base point contributes $e_p^m(y)$,
while changing the coordinate vector contributes $e_p^m(x)$,
since $\Exp_q^m(0)=q$ and $d(\Exp_p^m)_0=e_p^m$.
Applying $d((\Exp_p^m)^{-1})_p=(e_p^m)^{-1}$ gives $x+y$.

Taylor's formula with integral remainder therefore gives
	\begin{equation}
		\Phi(t,p,x)
		=
		t(x+y)
		+
		t^2\int_0^1(1-s)\,
		\partial_t^2\Phi(st,p,x)\,ds.
	\end{equation}
	Taking $t=\lambda_n^{-1/2}$ and differentiating in $x$ up to order $r$
	gives \eqref{Aest}, since the derivatives of $\partial_t^2\Phi$ are
	uniformly bounded on a fixed compact set.
	
To pass from the coordinate comparison to the values of \(F\), we use the tightness established in Lemma~\ref{exist}. Given
	$\varepsilon>0$, we can choose a compact set
	$K_\varepsilon\subset C^\infty(\mathbb{R}^d)$ such that
	\begin{equation}
		\mu_m^n(K_\varepsilon)\geq1-\varepsilon
	\end{equation}
	for all $n$, and set
	\begin{equation}
		G_{n,\varepsilon}:=
		\{p\in U_m:\psi_{p,m}^n\in K_\varepsilon\}.
	\end{equation}
	Thus
	\begin{equation}
		\Vol_g(U_m\setminus G_{n,\varepsilon})
		\leq
		\varepsilon\,\Vol_g(U_m).
	\end{equation}
	
	We claim that, with the supremum understood as zero
	when the indexing set is empty, we have
	\begin{equation}\label{Funiform}
		\sup_{p\in P\cap G_{n,\varepsilon}}
		\left|
		F(\psi_{Q_{n,y}(p),m}^n)
		-
		F(\tau_y\psi_{p,m}^n)
		\right|
		\longrightarrow0.
	\end{equation}
	Otherwise, along a subsequence there would be
	$p_n\in P\cap G_{n,\varepsilon}$ for which the difference in \eqref{Funiform} 
	stays bounded away from zero. By compactness of
	$K_\varepsilon$, we may assume that
	$f_n:=\psi_{p_n,m}^n\to f$ in $C^\infty(\mathbb R^d)$.
	For each fixed $L$, \eqref{Aest} places
	$A_n(p_n,\overline{B_L})$ in a fixed ball for all large enough $n$. Hence the
	chain rule and \eqref{shift} give
	\[
	\psi_{Q_{n,y}(p_n),m}^n
	=f_n\circ A_n(p_n,\cdot)
	\longrightarrow\tau_yf
	\]
	in $C^r(\overline{B_L})$ for every $L$ and $r$,
	and therefore in $C^\infty(\mathbb R^d)$.
	Also, $\tau_yf_n\to\tau_yf$. Continuity of $F$
	contradicts the choice of $p_n$.

For large $n$, the fields $\psi_{Q_{n,y}(p),m}^n$ are
defined for $p$ near $P$. After multiplication by $a$,
the expressions below are extended by zero outside
that neighborhood. Consequently,
\begin{equation}\label{first}
		\frac1{\Vol_g(U_m)}
		\int_{U_m}
		a(p)
		\left[
		F(\tau_y\psi_{p,m}^n)
		-
		F(\psi_{Q_{n,y}(p),m}^n)
		\right]\,dV_g(p)
		\longrightarrow0,
	\end{equation}
	because the integrand is uniformly small on
	$P\cap G_{n,\varepsilon}$ and bounded by $2\|F\|_\infty$ on its
	complement; we first let $n\to\infty$ and then $\varepsilon\to0$.
	
By \eqref{first}, it remains to replace the displaced center \(Q_{n,y}(p)\) by \(p\) in the average. We do this by a change of variables. Choose
	$\chi\in C_c^\infty(U_m)$ with $\chi=1$ on a neighborhood of $P$, and
	extend
	\begin{equation}
		Y(p):=\chi(p)e_p^m(y)
	\end{equation}
	by zero to a smooth vector field on $M$. Set
	\begin{equation}
		Q_n(p):=
		\exp_p\left(\frac{Y(p)}{\sqrt{\lambda_n}}\right).
	\end{equation}
	Then $Q_n=Q_{n,y}$ on $P$ and
	\begin{equation}
		Q_n\longrightarrow\id_M
		\qquad\text{in }C^1(M).
	\end{equation}
Since $M$ is compact and $C^1$ diffeomorphisms form
a $C^1$-open subset of $C^1(M,M)$, $Q_n$ is a
diffeomorphism for all sufficiently large $n$. If $J_n$ denotes its Jacobian with respect to $dV_g$, then
	\begin{equation}
		\|J_n-1\|_\infty\longrightarrow0.
	\end{equation}
	
	Extend $a$ by zero to $M$, and set
	$h_n(q):=F(\psi_{q,m}^n)$ on $U_m$ and $h_n(q):=0$
	elsewhere. Since $\overline{P}\subset U_m$ and $Q_n\to\id_M$
	uniformly, $\overline{Q_n(P)}\subset U_m$ for large $n$.
	Thus $h_n(Q_n(p))=F(\psi_{Q_{n,y}(p),m}^n)$ wherever
	$a(p)\neq0$. Consequently, 	$a(p)F(\psi_{Q_{n,y}(p),m}^n)=a(p)h_n(Q_n(p))$, for any $p\in M$. The change of variables
	$q=Q_n(p)$ gives
	\begin{equation}
		\int_M a(p)h_n(Q_n(p))\,dV_g(p)
		=
		\int_M
		\left(\frac{a}{J_n}\right)(Q_n^{-1}(q))
		h_n(q)\,dV_g(q).
	\end{equation}
	Moreover,
	\begin{equation}
		\left\|
		\left(\frac{a}{J_n}\right)\circ Q_n^{-1}-a
		\right\|_\infty
		\longrightarrow0,
	\end{equation}
	because $Q_n^{-1}\to\id_M$ uniformly,
	$J_n\to1$ uniformly, and $a$ is smooth. Since
	$|h_n|\leq\|F\|_\infty$, it follows that
	\begin{equation}\label{second}
		\frac1{\Vol_g(U_m)}
		\int_{U_m}
		a(p)
		\left[
		F(\psi_{Q_{n,y}(p),m}^n)
		-
		F(\psi_{p,m}^n)
		\right]\,dV_g(p)
		\longrightarrow0.
	\end{equation}
	
	Adding \eqref{first} and \eqref{second}, we obtain
	\begin{equation}\label{cutoff}
		\frac1{\Vol_g(U_m)}
		\int_{U_m}
		a(p)
		\left[
		F(\tau_y\psi_{p,m}^n)-F(\psi_{p,m}^n)
		\right]\,dV_g(p)
		\longrightarrow0.
	\end{equation}
	Finally,
	\begin{equation}
		\left|
		\int_{U_m}
		(1-a(p))
		\left(
		F(\tau_y\psi_{p,m}^n)-F(\psi_{p,m}^n)
		\right)\,dV_g(p)
		\right|
		\leq
	2\|F\|_\infty \,	\|1-a\|_{L^1(U_m)}.
	\end{equation}
	By regularity of the Riemannian volume measure, $a$ can be chosen so
	that the last $L^1$ norm is arbitrarily small. Together with
	\eqref{cutoff}, this proves \eqref{goal}.
	
	Taking $n=n_k$ and using $\mu_m^{n_k}\Rightarrow\mu_m$, we conclude that
	\begin{equation}
		\mu_m(\tau_yF)=\mu_m(F)
	\end{equation}
	for every $y\in\mathbb{R}^d$ and
	$F\in C_b(C^\infty(\mathbb{R}^d))$. Hence
	$\tau_y\mu_m=\mu_m$.
\end{proof}

We finish with a restriction property of the local weak limit
formulation, which is only used in Section~\ref{sec6}
to describe the distribution of localization points.

 \begin{lemma}\label{subset}
	Let $U\subset U_m$ be a nonempty open set, with $U_m$ as in
	Lemma~\ref{frames}. If the local fields $\psi_{p,m}^n$ converge
	in law to $\Psi_{\mathrm{RWM}}$ when $p$ is uniform on $U_m$ (i.e. Berry's property holds with the local formulation on $U_m$),
	then the same holds when $p$ is uniform on $U$, using the
	restriction of the same frame. 
\end{lemma}A proof of this fact can be found in \cite[Proposition~5.1]{Ingremeau2021} for the case of a domain of $M=\mathbb{R}^d$ and can be easily modified for our case as follows: 
\begin{proof}
	If $\Vol_g(U)=\Vol_g(U_m)$ then the result is trivial. If $0<\Vol_g(U)<\Vol_g(U_m)$, we have that, for any $F\in C_b\left(C^\infty\left(\mathbb{R}^d\right)\right)$, we can split the integral as \begin{equation}
		\begin{aligned}
			\frac{1}{\Vol_g(U_m)}\int_{U_m}F(\psi^n_{p,m})dV_g(p)&=\frac{1}{\Vol_g(U_m)}\int_{U}F(\psi^n_{p,m})\,dV_g(p)\\&+\frac{1}{\Vol_g(U_m)}\int_{U_m\backslash U}F(\psi^n_{p,m})\,dV_g(p).	
		\end{aligned}
	\end{equation}
	The tightness established in Lemma~\ref{exist} also gives tightness of the laws obtained by restricting the point
	to $U$ and $U_m\setminus U$. Starting from any subsequence,
	we may therefore extract a further subsequence $(n_k)_{k\in\mathbb{N}}$
	and probability measures $\mu_U$ and $\mu_U^c$ such that,
	for every $F\in C_b(C^\infty(\mathbb R^d))$, we have \begin{equation}
		\mu_U(F):=\lim_{k \rightarrow \infty}\frac{1}{\Vol_g(U)}\int_{U}F(\psi^{n_k}_{p,m})\,dV_g(p),
	\end{equation}
	and 
	\begin{equation}
		\mu_U^c(F):=\lim_{k \rightarrow \infty}\frac{1}{\Vol_g(U_m\backslash U)}\int_{U_m\backslash U}F(\psi^{n_k}_{p,m})\,dV_g(p).
	\end{equation}By hypothesis, we have then that \[\mu_{\mathrm{RWM}}=\frac{\Vol_g(U)}{\Vol_g(U_m)}\mu_U+\frac{\Vol_g(U_m\backslash U)}{\Vol_g(U_m)}\mu_U^c.\]
	
	To prove translation invariance, let $E$ be either $U$ or
	$U_m\setminus U$, and denote the corresponding limit measure
	by $\mu_E$. Fix $y\in\mathbb R^d$ and
	$F\in C_b(C^\infty(\mathbb R^d))$.
	For $a\in C_c^\infty(U_m)$ with $0\leq a\leq1$,
	\eqref{cutoff} shows that the integral with $a$
	tends to zero. Subtracting it from the integral over $E$
	and using the subsequence $(n_k)_{k\in\mathbb{N}}$ gives
	\[
	|\mu_E(\tau_yF)-\mu_E(F)|
	\leq
	\frac{2\|F\|_\infty}{\Vol_g(E)}
	\|\mathbf1_E-a\|_{L^1(U_m)}.
	\]
	By regularity of Riemannian volume, we can approximate
	$\mathbf1_E$ in $L^1$ by such functions $a$. Thus
	$\mu_E(\tau_yF)=\mu_E(F)$, proving translation invariance
	of $\mu_U$ and $\mu_U^c$.
	
	By Section~\ref{sec:rwm}, $\mu_{\mathrm{RWM}}$ is an extreme
	point of the set of translation-invariant probability measures.
	The decomposition above therefore forces
	$\mu_U=\mu_U^c=\mu_{\mathrm{RWM}}$.
	
	Since the starting subsequence was arbitrary, every
	subsequence admits a further subsequence whose restricted
	law converges to $\mu_{\mathrm{RWM}}$. Thus the full
	sequence of laws on $U$ converges to $\mu_{\mathrm{RWM}}$. \begin{equation}
		\mu_U(F)=\lim_{n \rightarrow \infty}\frac{1}{\Vol_g( U)}\int_{ U}F(\psi^{n}_{p,m})\,dV_g(p)=\mu_{\mathrm{RWM}}(F).
	\end{equation}
\end{proof}
\subsection{Benjamini--Schramm convergence of manifolds}

Before introducing the formulation of Berry's conjecture from \cite{ABLM}, we briefly recall the space of pointed Riemannian manifolds underlying Benjamini--Schramm convergence. Let
\begin{equation}
	\widetilde{\mathcal{M}^d}
	=
	\left\{
	(M,p):
	\begin{array}{l}
		M \text{ is a connected, complete, smooth Riemannian}\\
		d\text{-manifold and }p\in M
	\end{array}
	\right\},
\end{equation}
and define
\begin{equation}
	\mathcal{M}^d
	=
	\widetilde{\mathcal{M}^d}/\text{ pointed isometries}.
\end{equation}
Thus $[M,p]=[N,q]$ if there exists an isometry $S:M\to N$ such that
$S(p)=q$. The manifolds in $\mathcal{M}^d$ are not required to be compact.

The space $\mathcal{M}^d$ is endowed with the \emph{smooth topology}. Informally, two pointed manifolds are close when large neighborhoods of their base points can be identified by maps that are close to isometries in the $C^\infty$ sense. More precisely,
\[
[M_i,p_i]\longrightarrow[M,p]
\]
if there exist $R_i\to\infty$ and smooth embeddings
\begin{equation}
	f_i:B_g(p,R_i)\longrightarrow M_i,
\end{equation}
such that
\begin{equation}
	f_i(p)=p_i\qquad\text{and}\qquad
	f_i^*g_i\longrightarrow g,
\end{equation}
in the $C^\infty$-topology on compact subsets. We refer to
\cite[$\S$A.1--A.2]{BS} for further details. With this topology,
$\mathcal{M}^d$ is a Polish space.
\begin{example}\label{examp}
	Let $M_0=(M,g)$ be a closed $d$-dimensional Riemannian manifold and let
	$M_r=(M,r^2g)$ for $r>0$. Then, for every $p\in M$,
	\begin{equation}
		[M_r,p]\longrightarrow[\mathbb{R}^d,0]
		\qquad\text{as }r\to\infty
	\end{equation}
	in the smooth topology. Indeed, normal coordinates centered at $p$
	identify $0$ with $p$, and
	\[
	g_{ij}(x)=\delta_{ij}+O(|x|^2).
	\]
	After the rescaling $x\mapsto x/r$, the pulled-back metric has
	coefficients $g_{ij}(x/r)$ and converges in $C^\infty$ on every
	fixed Euclidean ball to the flat metric. Since $M$ is compact,
	this convergence is uniform in the choice of $p$.
\end{example}

If $M$ is compact, its normalized Riemannian volume measure induces a probability measure $\mu_M$ on $\mathcal{M}^d$ through the map
\begin{equation}\label{measure1}
	\begin{aligned}
	\iota_M:	M &\longrightarrow \mathcal{M}^d,\\
		p &\longmapsto \iota_M(p):=[M,p].
	\end{aligned}
\end{equation}
Since $\mathcal{M}^d$ is Polish, we use the notion of weak convergence introduced in Section~\ref{sec2}: for probability measures $\mu_n,\mu$ on $\mathcal{M}^d$,
\begin{equation}
	\mu_n\Rightarrow\mu
\end{equation}
means that $\mu_n(F)\to\mu(F)$ for every $F\in C_b(\mathcal{M}^d)$.

\begin{definition}[Benjamini--Schramm convergence for manifolds]
	A sequence $(N_n)_{n\in\mathbb{N}}$ of compact, connected, smooth Riemannian $d$-manifolds is said to converge in the Benjamini--Schramm sense, or BS-converge, if there exists a probability measure $\mu$ on $\mathcal{M}^d$ such that
	\begin{equation}
		\mu_{N_n}\Rightarrow\mu.
	\end{equation}
\end{definition}

In particular, if $M_n=(M,\lambda_n g)$ with $\lambda_n\to\infty$, then Example~\ref{examp} gives
\begin{equation}
	\mu_{M_n}\Rightarrow\delta_{[\mathbb{R}^d,0]}.
\end{equation}
This elementary example is the geometric part of the Benjamini--Schramm formulation used below; the next subsection incorporates the eigenfunctions into the limiting objects.

\subsection{The Berry property in Benjamini--Schramm form}

We now introduce the decorated version of Benjamini--Schramm convergence used in \cite{ABLM}. Consider 
\begin{equation}
	\widetilde{\mathcal{E}^d}
	=
	\left\{
	(M,p,\phi):
	\begin{array}{l}
		M \text{ connected, complete, smooth Riemannian } d\text{-manifold},\\
		p\in M,\ \phi\in C^\infty(M)
	\end{array}
	\right\},
\end{equation}
and let
\begin{equation}
	\mathcal{E}^d
	=
	\widetilde{\mathcal{E}^d}/\text{ pointed isometries preserving the function}.
\end{equation}
Thus $[M,p,\phi]=[N,q,\varphi]$ if there exists an isometry
$S:M\to N$ such that $S(p)=q$ and $\varphi\circ S=\phi$.

The space $\mathcal{E}^d$ is endowed with the smooth topology extending the one on $\mathcal{M}^d$: in addition to smooth convergence of the pointed manifolds, the functions and all their derivatives converge on compact sets under the corresponding embeddings. By \cite[Proposition 8]{ABLM}, $\mathcal{E}^d$ is a Polish space. We therefore use the notation
\begin{equation}
	\mu_n\Rightarrow\mu
\end{equation}
introduced in Section \ref{sec2} for weak convergence of Borel probability measures on $\mathcal{E}^d$.

If $M$ is compact and $\phi\in C^\infty(M)$, consider the map
\begin{equation}\label{rootmap}
	\begin{aligned}
		\iota_{M,\phi}:M &\longrightarrow \mathcal{E}^d,\\
		p &\longmapsto 		\iota_{M,\phi}(p):=[M,p,\phi].
	\end{aligned}
\end{equation}
The push-forward of the normalized Riemannian volume measure under this map defines a probability measure $\mu_{M,\phi}$ on $\mathcal{E}^d$, namely
\begin{equation}\label{measure2}
	\mu_{M,\phi}
	=
	\left(		\iota_{M,\phi}\right)_*
	\frac{dV_g}{\Vol_g(M)}.
\end{equation}
Thus only the base point is chosen at random, while the manifold and the function remain fixed.

\begin{definition}[BS-convergence of compact manifolds with functions]
	A sequence $(N_n,\phi_n)$ of compact connected smooth Riemannian $d$-manifolds endowed with smooth functions is said to Benjamini--Schramm converge, or BS-converge, to a probability measure $\mu$ on $\mathcal{E}^d$ if it holds that
	\begin{equation}
		\mu_{N_n,\phi_n}\Rightarrow\mu.
	\end{equation}
\end{definition}

We now specialize this construction to the eigenfunctions considered in this context. To do so, let
\begin{equation}
M_n:=(M,g_n),\qquad g_n:=\lambda_n g.
\end{equation}
Thus distances are multiplied by $\sqrt{\lambda_n}$ and
Riemannian volume by $\lambda_n^{d/2}$. Then
\begin{equation}
	\Delta_{\lambda_n g}=\lambda_n^{-1}\Delta_g,
\end{equation}
so the same function $\psi_n$, satisfying
$-\Delta_g\psi_n=\lambda_n\psi_n$, satisfies
\begin{equation}
	-\Delta_{\lambda_n g}\psi_n=\psi_n
\end{equation}
on $M_n$. By Example~\ref{examp}, the underlying pointed manifolds
converge to $[\mathbb R^d,0]$, uniformly in the base point.
Since smooth convergence of decorated manifolds includes
convergence of the underlying metrics, every weak limit
of $\mu_{M_n,\psi_n}$ is concentrated on Euclidean
decorated manifolds. The equation
$-\Delta_{\lambda_n g}\psi_n=\psi_n$ also passes to smooth
limits, so these have the form $[\mathbb R^d,0,f]$
with $f\in\mathrm{FP}$.

There is, however, an important quotient in this description. If $S\in O(d)$, then
\begin{equation}
	[\mathbb{R}^d,0,f]
	=
	[\mathbb{R}^d,0,f\circ S].
\end{equation}
Thus the BS limit records a Euclidean function only up to an orthogonal change of coordinates. Let
\begin{equation}
	\iota:C^\infty(\mathbb{R}^d)\longrightarrow\mathcal{E}^d,
	\qquad
	\iota(f):=[\mathbb{R}^d,0,f].
\end{equation}
The map $\iota$ is continuous, using the identity maps as
comparison embeddings, but is not injective:
$\iota(f)=\iota(h)$ precisely when $h=f\circ S$ for some
$S\in O(d)$. Before introducing the BS formulation of the model, we define the corresponding probability
measure on $\mathcal E^d$ by 
\begin{equation}\label{muBerry}
	\mu_{\mathrm{Berry}}
	:=
	\iota_*\mu_{\mathrm{RWM}}.
\end{equation}

\begin{definition}[The Berry property in BS form]\label{C2}
	Let $(\psi_n)_{n\in\mathbb{N}_0}$ be a real Laplace eigenbasis on $(M,g)$, with
	\begin{equation}
	-\Delta_g\psi_n=\lambda_n\psi_n,	\qquad	\|\psi_n\|_{L^2(M,dV_g)}^2=\Vol_g(M),	\end{equation}
	and let $M_n=(M,\lambda_n g)$. We say that $(\psi_n)_{n\in\mathbb{N}}$ satisfies the Berry property in Benjamini--Schramm form if
	\begin{equation}
		\mu_{M_n,\psi_n}\Rightarrow\mu_{\mathrm{Berry}}
		\qquad\text{as }n\to\infty.
	\end{equation}
\end{definition}

As in the local weak limit formulation, the question of which dynamical assumptions imply this property is separate from the equivalence studied here. One formulation of the random-wave conjecture predicts that suitable chaotic dynamics, for instance mixing geodesic flow, should imply the Berry property above. This conjecture is not studied in this paper.

\begin{remark}
	The essential difference between the two formulations is that the BS observation forgets the orthonormal coordinates used. Consequently, on the Euclidean limit it identifies $f$ with $f\circ S$ for $S\in O(d)$, whereas the local weak limit formulation retains the full function in the chosen frame. In the proof of the equivalence, we first relate the BS law to the corresponding local law averaged over $O(d)$; translation invariance and ergodicity are then used to recover the individual local laws. This is the reason for introducing the framed version of Benjamini--Schramm convergence in the next section.
\end{remark}
\section{An intermediate step: Framed Benjamini--Schramm convergence}
\label{sec:framed} 
Keeping an orthonormal frame allows us to compare the geometric
and local function formulations in the same coordinates.
We use this intermediate formulation to prove both
implications of Theorem~\ref{main}. 

It is often convenient to supplement the basepoint of a pointed manifold
$(M,p)$ with additional local data, such as an orthonormal basis for $T_pM$. We consider $\mathcal{F}M\longrightarrow M$, the
bundle of orthonormal frames. A pointed framed Riemannian manifold is a
pair $(M,(p,e_p))$, where $e_p:\mathbb{R}^d\longrightarrow T_pM$ is an
orthonormal frame. As before, we consider connected, complete, smooth
Riemannian manifolds and define
\begin{equation}
	\mathcal{F}\mathcal{M}^d
	:=
	\left.
	\left\{
	\begin{array}{c}
		\text{framed, connected, complete} \\
		\text{Riemannian \(d\)-manifolds } (M,(p,e_p))
	\end{array}
	\right\}
	\middle/
	\text{framed isometries.}
	\right.
\end{equation}
Thus $[M,(p,e_p)]=[N,(q,e_q)]$ if there is an isometry $S:M\to N$
such that $S(p)=q$ and $dS_p\circ e_p=e_q$.

 Smooth convergence is defined as for pointed manifolds, with the additional
 requirement that the comparison embeddings preserve the frames. More
 precisely, $(M_i,(p_i,e_{p_i}))\to(M,(p,e_p))$ if, for every $R>0$, there
 is a relatively compact open set $U\subset M$ containing
 $\overline{B_g(p,R)}$ and, for all sufficiently large $i$, embeddings
 $f_i:U\to M_i$ such that $f_i(p)=p_i,$ $df_i(p)\circ e_p=e_{p_i},$ and
 \begin{equation}
 	f_i^*g_i\longrightarrow g\quad\text{in }C^\infty\text{ on compact subsets of }U.
 \end{equation}
 By \cite[Theorem A.11]{BS}, $\mathcal{F}\mathcal{M}^d$ admits a complete,
 separable metric inducing smooth convergence. 
 
 We can also provide these objects with the extra data of a smooth function,
as in the previous subsection, by defining

\begin{equation}
	\mathcal{F}\mathcal{E}^d
	=
	\left.
	\left\{
	(M,(p,e_p),\phi)
	\;\middle|\;
	\substack{
		M \text{ connected, complete, smooth, Riemannian}\\
		d\text{-manifold},\ (p,e_p)\in\mathcal{F}M,\\
		\phi:M\to\mathbb{R}\text{ smooth}
	}
	\right\}
	\middle/
	\substack{\text{framed isometries}\\\text{preserving function}}
	\right..
\end{equation}
Here, we are now considering the quotient with respect to framed isometries that also preserve the function. In addition to
framed smooth convergence, we require $\phi_i\circ f_i\to\phi$ in
$C^\infty$ on compact subsets of $U$. We equip this space with the resulting smooth topology and its Borel $\sigma$-algebra. 

The metric construction of \cite[Theorem~A.11]{BS} extends
to this space by including control of the function and
its derivatives, as in \cite[Proposition~8]{ABLM}.
The same completeness and separability arguments then
show that $\mathcal F\mathcal E^d$ is Polish.

For compact $M$, we define a probability measure on $\mathcal{F}M$ by
choosing $p$ with normalized Riemannian volume and then choosing an
orthonormal frame at $p$ with Haar probability. Its push-forward under
\begin{equation}\label{measure3}
	\begin{aligned}
		\mathcal{F}M&\longrightarrow\mathcal{F}\mathcal{E}^d,\\
		(p,e_p)&\longmapsto[M,(p,e_p),\phi]
	\end{aligned}
\end{equation}
is denoted by $\mu_{M,\phi}'$. Explicitly, for
$F\in C_b(\mathcal{F}\mathcal{E}^d)$,
\begin{equation}\label{eq:framed-sampling}
	\mu_{M,\phi}'(F)
	=\frac{1}{\Vol_g(M)}\int_M\int_{O(d)}
	F([M,(p,e_p\circ R),\phi])\,dR\,dV_g(p),
\end{equation}
where $dR$ denotes Haar probability on $O(d)$. The frame $e_p$ is chosen on local trivializations, for instance those defined in \eqref{trivialframe}, and the inner integral is independent of this choice by Haar invariance.

We can therefore introduce the following natural definition.
\begin{definition}[FBS-convergence of compact manifolds with functions]
 A sequence $(N_n,\phi_n)$ of compact connected smooth Riemannian
$d$-manifolds endowed with smooth functions is said to converge in the Framed
Benjamini--Schramm sense, or FBS-converge, to a Borel probability
measure $\mu'$ on $\mathcal{F}\mathcal{E}^d$ if $\mu_{N_n,\phi_n}'\Rightarrow\mu'$. This means that \begin{equation}
	\mu_{N_n,\phi_n}'(F)\longrightarrow\mu'(F)
	\qquad\text{for every }F\in C_b(\mathcal{F}\mathcal{E}^d).
\end{equation}
\end{definition}

\begin{example}\label{exampleframe}
	We can extend Example~\ref{examp} to this context. Recall that
	$M_n=(M,g_n)$ with $g_n=\lambda_n g$, where we consider the positive
	eigenvalues tending to infinity. If $e_p$ is $g$-orthonormal, then
	$\lambda_n^{-1/2}e_p$ is $g_n$-orthonormal, and the maps
	\[
	x\longmapsto\exp_p\left(\frac{e_p(x)}{\sqrt{\lambda_n}}\right)
	\]
send $0$ to $p$, and their differentials at $0$ send the
standard frame $e^c$ to $\lambda_n^{-1/2}e_p$. Hence
	\begin{equation}
		[M_n,(p,\lambda_n^{-1/2}e_p)]
		\longrightarrow[\mathbb{R}^d,(0,e^c)],
	\end{equation}
	where $e^c$ is the standard frame. With the eigenfunctions added, we
	are therefore interested in the framed decorated objects
	\begin{equation}\label{sets2}
		[M_n,(p,\lambda_n^{-1/2}e_p),\psi_n].
	\end{equation}
	Unlike the unframed case, the map
	\begin{equation}\label{eq:framed-inclusion}
		\iota':C^\infty(\mathbb{R}^d)\longrightarrow\mathcal{F}\mathcal{E}^d,
		\qquad \iota'(f):=[\mathbb{R}^d,(0,e^c),f]
	\end{equation}
	is injective: a Euclidean isometry fixing the origin and the standard
	frame is the identity. It is also continuous, as follows by using the
	identity maps as comparison embeddings. Thus the framed Euclidean
	objects retain actual functions, rather than functions modulo rotations.
\end{example}
We similarly define
\begin{equation}\label{eq:framed-Berry}
	\mu_{\mathrm{Berry}}':=\iota'_*\mu_{\mathrm{RWM}}.
\end{equation}
Throughout the paper, $\mu_{\mathrm{RWM}}$ denotes the
random wave law on $C^\infty(\mathbb R^d)$, while
$\mu_{\mathrm{Berry}}$ and $\mu_{\mathrm{Berry}}'$
denote its images on $\mathcal E^d$ and
$\mathcal F\mathcal E^d$, respectively.

We now define the framed Berry's property and prove its
equivalence with Definition~\ref{C2}.
 \begin{definition}[The Berry property in FBS-form]\label{fbs}
Let $(M,g)$ be a compact connected $d$-dimensional Riemannian
manifold and set $M_n=(M,\lambda_n g)$. Let $(\psi_n)_{n\in\mathbb{N}_0}$ be a real Laplace eigenbasis on $(M,g)$, normalized by
	$\|\psi_n\|_{L^2(M,dV_g)}^2=\Vol_g(M)$. We say that $(\psi_n)_{n\in\mathbb{N}}$ satisfies the Berry
	property in FBS form if
	\begin{equation}
		\mu_{M_n,\psi_n}'\Rightarrow\mu_{\mathrm{Berry}}'.
	\end{equation}
\end{definition}
\begin{lemma}\label{equiv}
	Let $M$ be a compact, connected, $d$-dimensional manifold and let $(\psi_n)_{n\in\mathbb{N}_0}$ be an orthogonal basis of $L^2(M)$ of eigenfunctions associated to $\lambda_n$ normalized by $\|\psi_n\|_{L^2(M,dV_g)}^2=\Vol_g(M)$. Then, the Berry property in FBS form is equivalent to the Berry property in Benjamini--Schramm form.
\end{lemma}
\begin{proof}
We start by showing that the Berry property in FBS form implies the
Berry property in BS form. Given
$\widetilde F\in C_b(\mathcal E^d)$, define
\[
F([N,(p,e_p),\phi]):=\widetilde F([N,p,\phi]).
\]
This functional is well defined because every framed isometry preserving the function is also a pointed isometry preserving the function. Moreover, it is continuous: forgetting the condition on the frames in the comparison embeddings gives the comparison required in
$\mathcal E^d$. 
	
	Throughout the proof, $dR$ denotes Haar probability measure on $O(d)$.
	Recall that the sets $U_m$ are pairwise disjoint and cover $M$ up to a
	set of volume zero. The definition of the framed sampling measure gives
	\[
	\begin{aligned}
		\mu_{M_n,\psi_n}'(F)
		&=\frac{1}{\Vol_g(M)}
		\sum_{m=1}^{m_{\max}}\int_{U_m}\int_{O(d)}
		F\left(
		\left[M_n,
		\left(p,\lambda_n^{-1/2}e_p^m\circ R\right),\psi_n\right]
		\right)\,dR\,dV_g(p)\\
		&=\frac{1}{\Vol_g(M)}
		\int_M\widetilde F([M_n,p,\psi_n])\,dV_g(p)\\
		&=\mu_{M_n,\psi_n}(\widetilde F).
	\end{aligned}
	\]
	Here $\lambda_n^{-1/2}e_p^m$ is orthonormal for the metric
	$g_n=\lambda_n g$, and the normalized volume measures of $g_n$ and $g$
	coincide. Similarly, by the definitions of the two Berry measures,
	\[
	\begin{aligned}
		\mu_{\mathrm{Berry}}'(F)
		&=\int_{C^\infty(\mathbb R^d)}
		F([\mathbb R^d,(0,e^c),f])\,d\mu_{\mathrm{RWM}}(f)\\
		&=\int_{C^\infty(\mathbb R^d)}
		\widetilde F([\mathbb R^d,0,f])\,d\mu_{\mathrm{RWM}}(f)\\
		&=\mu_{\mathrm{Berry}}(\widetilde F).
	\end{aligned}
	\]
	Consequently, the FBS hypothesis implies
	\[
	\mu_{M_n,\psi_n}(\widetilde F)
	=\mu_{M_n,\psi_n}'(F)
	\longrightarrow
	\mu_{\mathrm{Berry}}'(F)
	=\mu_{\mathrm{Berry}}(\widetilde F).
	\]
	Since $\widetilde F$ was arbitrary, this proves the first implication.
	
	For the converse, suppose that the Berry property holds in BS form.
	Given $F\in C_b(\mathcal F\mathcal E^d)$, define
	\begin{equation}\label{eq:frame-average}
		\widetilde F([N,p,\phi])
		:=\int_{O(d)}F([N,(p,e_p\circ R),\phi])\,dR,
	\end{equation}
	where $e_p$ is any orthonormal frame at $p$.
	
We first check that this definition does not depend on the choices.
The definition is independent of the chosen frame:
any other orthonormal frame of $T_pN$ has the form $e_p'=e_p\circ R_0$ for some $R_0\in O(d)$, and left invariance of Haar
measure gives
\[
\begin{aligned}
	\int_{O(d)}F([N,(p,e_p'\circ R),\phi])\,dR
	&=\int_{O(d)}F([N,(p,e_p\circ R_0R),\phi])\,dR\\
	&=\int_{O(d)}F([N,(p,e_p\circ R),\phi])\,dR.
\end{aligned}
\]
It is also independent of the representative, since a pointed isometry $S:N\to\widehat N$ preserving the decoration sends $e_p$ to
$dS_p\circ e_p$ and identifies the corresponding framed decorated
objects. Thus $\widetilde F$ is well defined. Since $dR$ is a
probability measure,
\[
|\widetilde F([N,p,\phi])|
\leq\int_{O(d)}
|F([N,(p,e_p\circ R),\phi])|\,dR
\leq\|F\|_\infty.
\]
Hence $\widetilde F$ is bounded, with
$\|\widetilde F\|_\infty\leq\|F\|_\infty$.

We next prove continuity. Since $\mathcal E^d$ is metrizable, it
suffices to consider a sequence
$[N_i,p_i,\phi_i]\to[N,p,\phi]$. Choose comparison embeddings
$f_i:B_g(p,R_i)\to N_i$, with $R_i\to\infty$, such that
\[
f_i(p)=p_i,\qquad f_i^*g_i\to g,\qquad
\phi_i\circ f_i\to\phi
\]
smoothly on compact subsets.

To compare the frame averages, we first choose orthonormal frames
at $p_i$ compatible with this convergence. Fix a $g$-orthonormal frame $e_p$.
The transported frame $df_i(p)\circ e_p$ need not be
$g_i$-orthonormal, but its pairwise inner products converge to those
of an orthonormal frame, since $f_i^*g_i\to g$. Applying
Gram--Schmidt with respect to $g_i$, we obtain an orthonormal frame
\[
e_i=df_i(p)\circ e_p\circ A_i.
\]
The Gram matrices of the original frames converge to the identity,
so the Gram--Schmidt formulas give $A_i\to I$.

The embeddings $f_i$ need not send $e_p$ exactly to $e_i$.
We therefore modify them near $p$, without changing their
convergence elsewhere. Choose a coordinate chart
$x:V\to W\subset\mathbb R^d$, from a relatively compact neighborhood of $p$, with
$x(p)=0$ and $dx_p\circ e_p=I$, and a function
$\chi\in C_c^\infty(W)$ equal to one near $0$.
Extend $\chi$ by zero to $\mathbb R^d$ and set
\[
H_i(z)=z+\chi(z)(A_i-I)z.
\]
Thus $H_i(z)=A_i z$ near the origin and $H_i(z)=z$ outside
$\supp\chi$. Since
$\|D(H_i-\id)\|_\infty\to0$, the contraction mapping
and inverse function theorems ensure that, for large $i$, $H_i$
is a smooth change of coordinates on $\mathbb R^d$.

Since $H_i$ fixes $\mathbb R^d\setminus W$, bijectivity implies
$H_i(W)=W$. We can therefore define
$\kappa_i=x^{-1}\circ H_i\circ x$ on $V$ and extend it by the
identity to the whole $N$. These diffeomorphisms satisfy
\[
\kappa_i(p)=p,\qquad
d\kappa_i(p)\circ e_p=e_p\circ A_i,\qquad
\kappa_i\to\id
\]
smoothly on compact subsets. The last convergence follows
directly from the formula for $H_i$: every derivative of
$H_i-\id$ tends to zero.

All $\kappa_i$ fix the complement of the same compact set
$K=x^{-1}(\supp\chi)$ and preserve $K$.
Consequently, for large $i$, we have $K\subset B_g(p,R_i)$, so $\kappa_i$
preserves this ball. Thus $\widehat f_i:=f_i\circ\kappa_i$
is defined on $B_g(p,R_i)$ and satisfies 
\[
d\widehat f_i(p)\circ e_p=e_i,\qquad
\widehat f_i^*g_i\to g,\qquad
\phi_i\circ\widehat f_i\to\phi
\]
smoothly on compact subsets. Indeed, $\kappa_i(L)\subset L\cup K$ for every compact $L$. Thus the claimed limits on $L$ follow by the chain rule from $\kappa_i\to\id$ and the original convergence on the fixed compact set $L\cup K$.

The same embeddings preserve the rotated frames, giving
\[
[N_i,(p_i,e_i\circ R),\phi_i]
\longrightarrow [N,(p,e_p\circ R),\phi]
\qquad\text{for every }R\in O(d).
\]
We may use $e_i$ in the definition of
$\widetilde F([N_i,p_i,\phi_i])$, because the average is
independent of the chosen frame. Continuity of $F$ gives
pointwise convergence of the integrands, which are bounded
by $\|F\|_\infty$. Since $dR$ is Haar probability, dominated
convergence and the definition of $\widetilde F$ yield
\[
\widetilde F([N_i,p_i,\phi_i])
\longrightarrow \widetilde F([N,p,\phi]),
\]
as required for the continuity.

Using the definition of $\widetilde F$ in the sampling formula, we obtain
\[
\begin{aligned}
	\mu_{M_n,\psi_n}'(F)
	&=\frac{1}{\Vol_g(M)}
	\sum_{m=1}^{m_{\max}}\int_{U_m}\int_{O(d)}
	F\left(
	\left[M_n,
	\left(p,\lambda_n^{-1/2}e_p^m\circ R\right),\psi_n\right]
	\right)\,dR\,dV_g(p)\\
	&=\frac{1}{\Vol_g(M)}
	\int_M\widetilde F([M_n,p,\psi_n])\,dV_g(p)\\
	&=\mu_{M_n,\psi_n}(\widetilde F).
\end{aligned}
\]
It remains to identify the limiting integral. The Euclidean isometry
$x\mapsto Rx$ sends the standard frame $e^c$ to $e^c\circ R$ and gives
\[
[\mathbb R^d,(0,e^c\circ R),f]
=[\mathbb R^d,(0,e^c),f\circ R].
\]
Therefore,
\[
\begin{aligned}
	\mu_{\mathrm{Berry}}(\widetilde F)
	&=\int_{C^\infty(\mathbb R^d)}\int_{O(d)}
	F([\mathbb R^d,(0,e^c\circ R),f])
	\,dR\,d\mu_{\mathrm{RWM}}(f)\\
	&=\int_{O(d)}\int_{C^\infty(\mathbb R^d)}
	F([\mathbb R^d,(0,e^c),f\circ R])
	\,d\mu_{\mathrm{RWM}}(f)\,dR\\
	&=\int_{C^\infty(\mathbb R^d)}
	F([\mathbb R^d,(0,e^c),f])\,d\mu_{\mathrm{RWM}}(f)\\
	&=\mu_{\mathrm{Berry}}'(F).
\end{aligned}
\]
The interchange of integrals is justified by Fubini's theorem:
the map $(R,f)\mapsto f\circ R$ is continuous in the smooth topology,
so the integrand is measurable, and its absolute value is bounded
by $\|F\|_\infty$ on a product of probability spaces.
The third equality follows from the orthogonal invariance of
$\mu_{\mathrm{RWM}}$.

Finally, the BS hypothesis gives
\[
\mu_{M_n,\psi_n}'(F)
=\mu_{M_n,\psi_n}(\widetilde F)
\longrightarrow
\mu_{\mathrm{Berry}}(\widetilde F)
=\mu_{\mathrm{Berry}}'(F).
\]
Since $F\in C_b(\mathcal F\mathcal E^d)$ was arbitrary, this proves
the Berry property in FBS form.
	\end{proof}
Thanks to this lemma, it remains to establish the equivalence
between the local weak limit and FBS formulations.
The framed sampling measure still averages over frames;
recovering the local laws in the prescribed frames $e_p^m$ requires translation invariance and ergodicity.

\subsection{Local functions and FBS convergence}

A chosen frame associates to each decorated manifold the function
$\phi\circ\exp_p\circ e_p$ on $\mathbb R^d$.
We first prove that this construction is continuous, and then
show that, for the rescaled manifolds considered here, convergence
of the laws of these functions is equivalent to FBS convergence.
\begin{lemma}\label{lem:framed-pullback}
	The map
	\begin{equation}\label{eq:framed-pullback}
		\begin{aligned}
			\mathcal Q:\mathcal F\mathcal E^d
			&\longrightarrow C^\infty(\mathbb R^d),\\
			[N,(p,e_p),\phi]
			&\longmapsto\phi\circ\exp_p\circ e_p
		\end{aligned}
	\end{equation}
	is well defined and continuous. Moreover, $\mathcal Q\circ\iota'$
	is the identity on $C^\infty(\mathbb R^d)$.
\end{lemma}

\begin{proof}
	Since $N$ is complete, geodesics starting at $p$ are defined for every
	initial velocity and time. Thus $\exp_p(e_p y)$ is defined for
	every $y\in\mathbb R^d$.	If $S:N\to\widetilde N$ is an isometry, uniqueness of geodesics with
	prescribed initial position and velocity gives
	\[
	S\circ\exp_p=\exp_{S(p)}\circ dS_p.
	\]
	If $S$ also preserves the frame and the decoration, composing this
	identity with $\widetilde\phi$ shows that both representatives give
	the same function. Hence $\mathcal Q$ is well defined.
	
The map $\mathcal Q([N,(p,e_p),\phi])(y)$ is obtained by following
the geodesic starting at $p$ with velocity $e_p y$ up to time one,
and then evaluating $\phi$ at its endpoint. To prove continuity,
we must show that small changes in the framed decorated manifold produce small
changes in this function and its derivatives on each fixed ball.
Since $\mathcal F\mathcal E^d$ is metrizable, it suffices to consider
a convergent sequence
\[
[N_i,(p_i,e_{p_i}),\phi_i]\longrightarrow[N,(p,e_p),\phi].
\]
To compare the functions obtained through $\mathcal Q$,
we first bring the metrics and decorations to the same neighborhood
of $p$. Fix $L>0$. Choose a relatively compact open set $U$ containing
$K:=\overline{B_g(p,L+1)}$. For sufficiently large $i$, the
definition of convergence provides embeddings $f_i:U\to N_i$ with
\[
f_i(p)=p_i,\qquad df_i(p)\circ e_p=e_{p_i},
\]
such that
\[
\widehat g_i:=f_i^*g_i\longrightarrow g,\qquad
\widehat\phi_i:=\phi_i\circ f_i\longrightarrow\phi
\]
smoothly on compact subsets of $U$.

We next compare the geodesics with the same initial position
and velocity for these metrics. For $|y|\leq L$, let
\[
c(t,y)=\exp_p^g(te_p y),\qquad 0\leq t\leq1.
\]
These geodesics have constant speed $|y|$, so
$\dist_g(p,c(t,y))\leq t|y|\leq L$, since $0\leq t\leq 1$. Their positions therefore remain
in the interior of $K$, and their positions and velocities
together range over a compact subset of $TU$.

The geodesics for $\widehat g_i$ satisfy the same initial
conditions and a differential equation whose coefficients
converge smoothly to those for $g$. Indeed, in local coordinates
the equation is
\[
\ddot z^a=-\Gamma^a_{bc}(z)\dot z^b\dot z^c,
\]
and the Christoffel symbols depend on the inverse metric and
its first derivatives.

We apply continuous dependence of the time-one flow
\cite[Theorem~4.1]{LewisZhang2023} to the first-order
geodesic system for position and velocity.
The compact set of initial data
$\{(p,e_p y):|y|\leq L\}$ lies in the open domain of
existence up to time one for $g$ on $U$.
We may therefore apply the theorem on a relatively
compact neighborhood whose closure lies in this domain.

It follows that, for large $i$, the $\widehat g_i$-geodesics
$\widehat c_i(t,y)$ with the same initial data also exist up
to time one, and their endpoints satisfy
\[
\widehat c_i(1,\cdot)\longrightarrow c(1,\cdot)
\]
smoothly on $\overline{B_L}$. The theorem also gives convergence of derivatives, so every derivative with respect to $y$
converges uniformly on $\overline{B_L}$.

This uniform convergence ensures that $\widehat c_i(1,y)\in K=\overline{B_g(p,L+1)}$
for large $i$. Both compositions therefore evaluate the
decorations on the fixed compact set $K$, where smooth
convergence and the chain rule give, for every integer $r\geq0$,
\[
\left\|
\widehat\phi_i\circ\widehat c_i(1,\cdot)
-\phi\circ c(1,\cdot)
\right\|_{C^r(\overline{B_L})}
\longrightarrow0.
\]

Finally, $f_i$ is an isometry onto its image for the metric
$\widehat g_i$ and sends $e_p$ to $e_{p_i}$. Uniqueness of
geodesics therefore gives
\[
f_i(\widehat c_i(1,y))
=\exp_{p_i}^{g_i}(e_{p_i}y).
\]
The functions in the preceding limit are precisely the images
under $\mathcal Q$ of the framed decorated manifolds.
Since $L$ and $r$ were arbitrary, this proves continuity.

Finally, in the canonical Euclidean frame the exponential map is the
identity, so we get $\mathcal Q(\iota'(f))=f$.
\end{proof}

For $H\in C_b(C^\infty(\mathbb R^d))$ and $R\in O(d)$, write
\begin{equation}\label{rot}
	H^R(f):=H(f\circ R),
\end{equation}
and define the measure on $C^\infty\left(\mathbb{R}^d\right)$ given by
\begin{equation}\label{eq:nu-framed}
	\nu_n:=\mathcal Q_*\mu_{M_n,\psi_n}'.
\end{equation}

We express $\nu_n$ as a volume-weighted average of the local
laws $\mu_m^n$, averaged over orthogonal changes of coordinates.
For $p\in U_m$, every $g_n$-orthonormal frame of $M_n$ has the form
$\lambda_n^{-1/2}e_p^m\circ R$, with $R\in O(d)$.
Constant rescaling of a metric leaves the Levi--Civita connection,
and hence the exponential map, unchanged. Thus
\[
\mathcal Q\left(
\left[M_n,
\left(p,\lambda_n^{-1/2}e_p^m\circ R\right),\psi_n\right]
\right)(y)
=
\psi_n\left(
\exp_p\left(\frac{e_p^m(Ry)}{\sqrt{\lambda_n}}\right)
\right)
=
\psi_{p,m}^n(Ry).
\]
Consequently, averaging over frames is the same as averaging
$\psi_{p,m}^n\circ R$ over $R\in O(d)$.
Moreover, $g_n$ and $g$ have the same normalized volume measure, since
both volume and total volume scale by $\lambda_n^{d/2}$.
As the sets $U_m$ partition $M$ up to a null set,
the definition of $\nu_n$ gives
\begin{equation}\label{eq:nu-local}
	\nu_n(H)
	=
	\sum_{m=1}^{m_{\max}}
	\frac{\Vol_g(U_m)}{\Vol_g(M)}
	\int_{O(d)}\mu_m^n(H^R)\,dR,
\end{equation}
where $dR$ is Haar probability measure on $O(d)$.

The following lemma reduces FBS convergence to
$\nu_n\Rightarrow\mu_{\mathrm{RWM}}$, which we then
compare with the local weak limit formulation.

\begin{lemma}\label{lem:reduction}
	For the rescaled sequence under consideration,
	\[
	\mu_{M_n,\psi_n}'\Rightarrow\mu_{\mathrm{Berry}}'
	\quad\Longleftrightarrow\quad
	\nu_n\Rightarrow\mu_{\mathrm{RWM}}.
	\]
\end{lemma}

\begin{proof}
	Suppose first that
	$\mu_{M_n,\psi_n}'\Rightarrow\mu_{\mathrm{Berry}}'$.
	For every $H\in C_b(C^\infty(\mathbb R^d))$, the functional
	$H\circ\mathcal Q$ is bounded and continuous by
	Lemma~\ref{lem:framed-pullback}. Since
	$\mu_{\mathrm{Berry}}'=\iota'_*\mu_{\mathrm{RWM}}$ by \eqref{eq:framed-Berry} and
	$\mathcal Q\circ\iota'=\id$, we obtain
	\[
	\nu_n(H)
	=\mu_{M_n,\psi_n}'(H\circ\mathcal Q)
	\longrightarrow
	\mu_{\mathrm{Berry}}'(H\circ\mathcal Q)
	=\mu_{\mathrm{RWM}}(H).
	\]
	
	Conversely, suppose that $\nu_n\Rightarrow\mu_{\mathrm{RWM}}$,
	and fix $F\in C_b(\mathcal F\mathcal E^d)$. 
We first prove uniform convergence of the rescaled metrics
to the Euclidean metric, and then use tightness to control
the local functions.
	
	For a $g$-orthonormal frame $e_p$ at $p$, set
	\[
	E_{n,p,e_p}(x):=
	\exp_p\left(\frac{e_px}{\sqrt{\lambda_n}}\right),
	\qquad
	f_{n,p,e_p}:=\psi_n\circ E_{n,p,e_p}.
	\]
	We first compare the rescaled and Euclidean metrics through
	$E_{n,p,e_p}$, uniformly in the point and frame.
	
	Fix $0<\rho<r_{\mathrm{inj}}$, and let $G(p,e_p,z)$
	be the coefficient matrix of $g$ in the normal coordinates
	$z\mapsto\exp_p(e_pz)$. It depends smoothly on $(p,e_p,z)$
	and satisfies $G(p,e_p,0)=I$. Compactness of $\mathcal FM$
	therefore gives uniform bounds for every derivative in $z$
	when $|z|\leq\rho$.	The coefficient matrix of $E_{n,p,e_p}^*g_n$ is then
	\[
	G\left(p,e_p,\frac{x}{\sqrt{\lambda_n}}\right),
	\]since the factor $\lambda_n$ in the metric cancels with the two
	factors $\lambda_n^{-1/2}$ from the differential of the
	coordinate change. 
	
	On each fixed ball $\overline{B_L}$,
	these coefficients converge uniformly to $I$ by the mean
	value theorem. Every derivative of positive order also
	converges uniformly to zero:
	\[
	\partial_x^\alpha
	G\left(p,e_p,\frac{x}{\sqrt{\lambda_n}}\right)
	=
	\lambda_n^{-|\alpha|/2}
	(\partial_z^\alpha G)
	\left(p,e_p,\frac{x}{\sqrt{\lambda_n}}\right)\rightarrow 0.
	\]
	Thus, writing $g_{\mathrm{eucl}}$ for the Euclidean metric, for
	every $L>0$ and integer $r\geq0$,
	\begin{equation}\label{eq:metric-conv}
		\sup_{(p,e_p)\in\mathcal FM}
		\|E_{n,p,e_p}^*g_n-g_{\mathrm{eucl}}\|_
		{C^r(\overline{B_L})}
		\longrightarrow0.
	\end{equation}
	For large $n$, $L/\sqrt{\lambda_n}<\rho$, so these maps are
	embeddings on a neighborhood of $\overline{B_L}$.
They send $0$ to $p$, their differentials at $0$ send
$e^c$ to $\lambda_n^{-1/2}e_p$, and they pull $\psi_n$
back to $f_{n,p,e_p}$.
	
	We now show that this geometric convergence gives a uniform
	comparison of the values of $F$ whenever the local functions
	belong to a fixed compact set. More precisely, for every compact
	$K\subset C^\infty(\mathbb R^d)$,
	\begin{equation}\label{eq:framed-compact-comparison}
		\sup_{\substack{(p,e_p)\in\mathcal FM\\ f_{n,p,e_p}\in K}}
		\left|
		F\left([M_n,(p,\lambda_n^{-1/2}e_p),\psi_n]\right)
		-F(\iota'(f_{n,p,e_p}))
		\right|
		\longrightarrow0,
	\end{equation}
	with the supremum understood to be zero when the indexing
	set is empty.
	
	Suppose otherwise. We could choose $n_j\to\infty$ and frames
	$(p_j,e_j)$ for which $f_{n_j,p_j,e_j}\in K$ and the displayed
	difference stays bounded away from zero. Compactness of $K$
	allows us to pass to a subsequence such that
	\[
	f_{n_j,p_j,e_j}\longrightarrow f
	\quad\text{in }C^\infty(\mathbb R^d)
	\]
	for some $f\in K$. Using $E_{n_j,p_j,e_j}$ as comparison
	embeddings, \eqref{eq:metric-conv} and the exact identities
	for the frames and functions give
	\[
	[M_{n_j},(p_j,\lambda_{n_j}^{-1/2}e_j),\psi_{n_j}]
	\longrightarrow
	[\mathbb R^d,(0,e^c),f]=\iota'(f).
	\]
	By continuity of $\iota'$, the Euclidean objects
	$\iota'(f_{n_j,p_j,e_j})$ converge to the same limit.
	Continuity of $F$ then forces the difference of its values
	to tend to zero, a contradiction. This proves
	\eqref{eq:framed-compact-comparison}.
	
To compare $\mu_{M_n,\psi_n}'(F)$ and $\nu_n(F\circ\iota')$,
we integrate the two values of $F$ over $p\in M$ with
normalized volume and over $g$-orthonormal frames $e_p$ with
Haar probability.
	
	Since $(\nu_n)$ converges weakly on the Polish space
	$C^\infty(\mathbb R^d)$, it is uniformly tight
	\cite[Section~5]{Billingsley1999}. This implies that, given $\varepsilon>0$, one can choose
	a compact $K$ with $\nu_n(K^c)<\varepsilon$ for every $n$.
	When $f_{n,p,e_p}\in K$, the comparison above makes the
	difference of the two values of $F$ uniformly small. When $f_{n,p,e_p}\notin K$, each value of $F$ is bounded
	in absolute value by $\|F\|_\infty$, so their difference is
	at most $2\|F\|_\infty$. These points and frames have total
	probability $\nu_n(K^c)<\varepsilon$, so their contribution
	to the averaged difference is at most
	$2\|F\|_\infty\varepsilon$. Therefore
	\[
	\limsup_{n\to\infty}
	\left|\mu_{M_n,\psi_n}'(F)-\nu_n(F\circ\iota')\right|
	\leq 2\|F\|_\infty\varepsilon.
	\]
	Since $\varepsilon$ is arbitrary, this difference tends to zero.
	Finally, $F\circ\iota'$ is bounded and continuous, so
	\[
	\nu_n(F\circ\iota')
	\longrightarrow
	\mu_{\mathrm{RWM}}(F\circ\iota')
	=\mu_{\mathrm{Berry}}'(F),
	\]
which proves the required FBS convergence.
\end{proof}

\section{Proof of Theorem~\ref{main}}\label{sec:proof-main}

By Lemma~\ref{equiv}, it remains to prove the equivalence
between the local weak limit and FBS formulations.
We prove the two implications separately.

\subsection{The local weak limit formulation implies
	the FBS formulation}\label{sec4}
	
Assume that the local weak limit formulation holds.
By Lemma~\ref{lem:reduction}, it is enough to show that
$\nu_n\Rightarrow\mu_{\mathrm{RWM}}$. Fix
$H\in C_b(C^\infty(\mathbb{R}^d))$. For every $m$ and every fixed
$R\in O(d)$, the local weak limit hypothesis and the orthogonal
invariance of $\mu_{\mathrm{RWM}}$ give
\begin{equation}
	\mu_m^n(H^R)\longrightarrow
	\mu_{\mathrm{RWM}}(H^R)=\mu_{\mathrm{RWM}}(H).
\end{equation}
Moreover, $|\mu_m^n(H^R)|\leq\|H\|_\infty$. Thus dominated convergence
with respect to Haar probability in \eqref{eq:nu-local} yields
\begin{equation}
	\begin{aligned}
		\nu_n(H)&\longrightarrow
		\sum_{m=1}^{m_{\max}}\frac{\Vol_g(U_m)}{\Vol_g(M)}
		\mu_{\mathrm{RWM}}(H)\\
		&=\mu_{\mathrm{RWM}}(H).
	\end{aligned}
\end{equation}
Lemma~\ref{lem:reduction} proves the FBS property, and Lemma~\ref{equiv}
then gives the BS property.

\subsection{The FBS formulation implies the local weak limit formulation}\label{sec5}

For the other implication, assume
$\mu_{M_n,\psi_n}'\Rightarrow\mu_{\mathrm{Berry}}'$.
Lemma~\ref{lem:reduction} gives
$\nu_n\Rightarrow\mu_{\mathrm{RWM}}$, which means that,
for every
$\widetilde G\in C_b(C^\infty(\mathbb{R}^d))$, it holds
\begin{equation}\label{eq:average-for-recovery}
	\sum_{m=1}^{m_{\max}}\frac{\Vol_g(U_m)}{\Vol_g(M)}
	\int_{O(d)}\mu_m^n(\widetilde G^R)\,dR
	\longrightarrow\mu_{\mathrm{RWM}}(\widetilde G).
\end{equation}
Put $G:=\widetilde G\circ\mathcal Q$
and consider
\begin{equation}
	J_n^{m,R}(G):=\mu_m^n(\widetilde G^R)
	=\frac{1}{\Vol_g(U_m)}\int_{U_m}
	\widetilde G(\psi_{p,m}^n\circ R)\,dV_g(p).
\end{equation}
We show that this convergence of the averaged laws forces
each local law to converge to $\mu_{\mathrm{RWM}}$.

The proof of Lemma \ref{exist} shows that each family
$\{\mu_m^n\}_{n\geq1}$ is tight. Hence the same is true after
restricting to any subsequence. Let $(n_j)_{j\in\mathbb{N}}$ be an arbitrary
subsequence of the original sequence. Since there are only finitely
many values of $m$, by successively extracting subsequences for
$m=1,\ldots,m_{\max}$ we obtain a common further subsequence, which
we denote by $(n_k)_{k\in\mathbb{N}}$, such that, for every $m$,
\begin{equation}
	\mu_m^{n_k}\Rightarrow\mu_m^\infty
\end{equation}
for some probability measure $\mu_m^\infty$ on
$C^\infty\left(\mathbb{R}^d\right)$. Therefore, for every
$R\in O(d)$,
\begin{equation}
	\lim_{k\to\infty}J_{n_k}^{m,R}(G)
	=\lim_{k\to\infty}\mu_m^{n_k}(\widetilde{G}^R)
	=\mu_m^\infty(\widetilde{G}^R).
\end{equation}
For each $R\in O(d)$, let $\mu_{m,R}^\infty$ be the
push-forward of $\mu_m^\infty$ under $f\mapsto f\circ R$.
Using the notation in \eqref{rot}, this means that
\[
\mu_{m,R}^\infty(F)=\mu_m^\infty(F^R),
\qquad F\in C_b(C^\infty(\mathbb R^d)).
\]

Moreover, by Lemma \ref{translation}, each $\mu_m^\infty$ is
translation invariant. We now check that the same is true for
$\mu_{m,R}^\infty$. For every $y\in\mathbb{R}^d$ and every
$F\in C_b\left(C^\infty\left(\mathbb{R}^d\right)\right)$,
one has
\begin{equation}\label{rotationtranslation}
	(\tau_yF)^R=\tau_{Ry}(F^R).
\end{equation}
Indeed, for every $f\in C^\infty\left(\mathbb{R}^d\right)$,
\begin{equation}
	\begin{aligned}
		(\tau_yF)^R(f)
		&=(\tau_yF)(f\circ R)
		=F\left(\tau_y(f\circ R)\right)\\
		&=F\left((\tau_{Ry}f)\circ R\right)
		=F^R(\tau_{Ry}f)
		=\tau_{Ry}(F^R)(f).
	\end{aligned}
\end{equation}
Hence, using the translation invariance of $\mu_m^\infty$,
\begin{equation}
	\begin{aligned}
		\mu_{m,R}^\infty(\tau_yF)
		&=\mu_m^\infty\left((\tau_yF)^R\right)
		=\mu_m^\infty\left(\tau_{Ry}(F^R)\right)\\
		&=\mu_m^\infty(F^R)
		=\mu_{m,R}^\infty(F).
	\end{aligned}
\end{equation}
Thus every $\mu_{m,R}^\infty$ is translation invariant. Moreover,
the map $R\mapsto\mu_{m,R}^\infty$ is continuous for the weak
topology on probability measures. Indeed, for every
$F\in C_b\left(C^\infty\left(\mathbb{R}^d\right)\right)$,
if $R_j\to R$ in $O(d)$, then $F^{R_j}(f)\to F^R(f)$ for every $f$,
while $|F^{R_j}(f)|\leq\|F\|_\infty$; dominated convergence therefore
gives
\[
\mu_m^\infty(F^{R_j})\longrightarrow\mu_m^\infty(F^R).
\]

For every fixed $m$ and $R$, we have
\begin{equation}
	\left|J_{n_k}^{m,R}(G)\right|\leq\left|\frac{1}{\Vol_g(U_m)}\int_{U_m}\widetilde{G}\left(\psi_{n_k}\left(\exp_p\left(\frac{e_p^m\left(R(\cdot)\right)}{\sqrt{\lambda_{n_k}}}\right)\right)\right)\, dV_g(p)\right|
	\leq\|\widetilde{G}\|_\infty.
\end{equation}
Moreover, $R\mapsto J_{n_k}^{m,R}(G)$ is continuous.
Indeed, for each $f\in C^\infty(\mathbb R^d)$, the map
$R\mapsto\widetilde G(f\circ R)$ is continuous and bounded
by $\|\widetilde G\|_\infty$, so dominated convergence
gives continuity after integration against $\mu_m^{n_k}$.

Since
\[
J_{n_k}^{m,R}(G)
\longrightarrow
\mu_{m,R}^\infty(\widetilde{G})
\]
for every fixed $R$, the dominated convergence theorem and the
hypothesis of FBS convergence give
\begin{equation}\label{convexdecomposition}
	\begin{aligned}
		\mu_{\mathrm{RWM}}(\widetilde{G})
		&=
		\lim_{k\to\infty}
		\mu_{M_{n_k},\psi_{n_k}}'(G)
		=
		\sum_{m=1}^{m_{\max}}
		w_m
		\int_{O(d)}
		\mu_{m,R}^\infty(\widetilde{G})\,dR.
	\end{aligned}
\end{equation}
Define the averaged probability measure
\begin{equation}
	\widetilde\mu_m^\infty
	:=
	\int_{O(d)}\mu_{m,R}^\infty\,dR.
\end{equation}
It is translation invariant, since every
$\mu_{m,R}^\infty$ is translation invariant. Since
$\widetilde{G}$ was arbitrary, equation
\eqref{convexdecomposition} gives the convex decomposition
\begin{equation}
	\mu_{\mathrm{RWM}}
	=
	\sum_{m=1}^{m_{\max}}
	w_m\widetilde\mu_m^\infty.
\end{equation}
Recall that $	w_m=\frac{\Vol_g(U_m)}{\Vol_g(M)}>0$ and $	\sum_{m=1}^{m_{\max}}w_m=1.$ Since $\mu_{\mathrm{RWM}}$ is ergodic with respect to translations
and all the weights $w_m$ are positive, Lemma \ref{int}, applied
to the finite probability space with masses $w_m$, gives
\begin{equation}\label{averageequalsrwm}
	\widetilde\mu_m^\infty
	=
	\mu_{\mathrm{RWM}}
	\qquad
	\text{for every }m=1,\ldots,m_{\max}.
\end{equation}

Fix now $m$. From \eqref{averageequalsrwm}, we have
\begin{equation}
	\mu_{\mathrm{RWM}}
	=
	\int_{O(d)}\mu_{m,R}^\infty\,dR.
\end{equation}
As observed above, the map
$R\mapsto\mu_{m,R}^\infty$ is measurable. Since all the measures
$\mu_{m,R}^\infty$ are translation invariant, Lemma \ref{int},
applied with normalized Haar measure on $O(d)$, yields
\begin{equation}
	\mu_{m,R}^\infty
	=
	\mu_{\mathrm{RWM}}
\end{equation}
for almost every $R\in O(d)$. Choose one such $R_0$. Then, for
every
$F\in C_b\left(C^\infty\left(\mathbb{R}^d\right)\right)$,
\begin{equation}\label{oneR}
	\mu_m^\infty(F^{R_0})
	=
	\mu_{m,R_0}^\infty(F)
	=
	\mu_{\mathrm{RWM}}(F).
\end{equation}
For an arbitrary
$\widetilde{H}\in
C_b\left(C^\infty\left(\mathbb{R}^d\right)\right)$,
we apply \eqref{oneR} with
$F=\widetilde{H}^{R_0^{-1}}$. Since
$(\widetilde{H}^{R_0^{-1}})^{R_0}=\widetilde{H}$, we obtain
\begin{equation}
	\mu_m^\infty(\widetilde{H})
	=
	\mu_{\mathrm{RWM}}
	\left(\widetilde{H}^{R_0^{-1}}\right)
	=
	\mu_{\mathrm{RWM}}(\widetilde{H}),
\end{equation}
where the last equality follows from the orthogonal invariance
of $\mu_{\mathrm{RWM}}$, recalled in Section~\ref{sec:rwm}. Thus
\begin{equation}
	\mu_m^\infty
	=
	\mu_{\mathrm{RWM}}
	\qquad
	\text{for every }m=1,\ldots,m_{\max}.
\end{equation}

We have proved that every subsequence $(n_j)_{j\in\mathbb{N}}$ admits a further
subsequence $(n_k)_{k\in\mathbb{N}}$ such that
$\mu_m^{n_k}\Rightarrow\mu_{\mathrm{RWM}}$ for every $m$.
This implies convergence of the full sequence. Indeed, if for some
$m_0$ and
$\widetilde{G}_0\in
C_b\left(C^\infty\left(\mathbb{R}^d\right)\right)$
the sequence $\mu_{m_0}^n(\widetilde{G}_0)$ did not converge to
$\mu_{\mathrm{RWM}}(\widetilde{G}_0)$, there would exist
$\varepsilon_0>0$ and a subsequence $(n_j)_{j\in\mathbb{N}}$ such that
\begin{equation}
	\left|
	\mu_{m_0}^{n_j}(\widetilde{G}_0)
	-
	\mu_{\mathrm{RWM}}(\widetilde{G}_0)
	\right|
	>
	\varepsilon_0
\end{equation}
for every $j\in\mathbb{N}$. The preceding argument, applied to this subsequence,
would provide a further subsequence along which the left-hand side
tends to zero, a contradiction. Therefore
\begin{equation}
	\lim_{n\to\infty}
	\mu_m^n(\widetilde{G})
	=
	\mu_{\mathrm{RWM}}(\widetilde{G})
\end{equation}
for every $m=1,\ldots,m_{\max}$ and every
$\widetilde{G}\in
C_b\left(C^\infty\left(\mathbb{R}^d\right)\right)$.
This proves the local weak limit formulation and completes
the proof of Theorem~\ref{main}.

\section{The inverse localization property}\label{sec6}

We conclude by relating the Berry property to the inverse localization property, a deterministic property describing the local flexibility of high-energy eigenfunctions. This connection has two complementary consequences: inverse localization transfers prescribed local features of Euclidean Helmholtz solutions to genuine eigenfunctions, with applications to nodal sets, while failure of inverse localization gives an obstruction to the Berry property.

The starting point is the relation between high-energy eigenfunctions and the Euclidean Helmholtz equation. After rescaling around a point at the wavelength scale $1/\sqrt{\lambda}$, the local behavior of a Laplace eigenfunction with eigenvalue $\lambda$ is governed, in the high-energy limit, by solutions of the Helmholtz equation
\begin{equation}\label{Helmh}
	\Delta\varphi+\varphi=0
	\qquad\text{in }\mathbb{R}^d.
\end{equation}
A partial converse, based on H\"ormander's local Weyl law
\cite{hormander1969spectral}, shows that any prescribed solution
of \eqref{Helmh} on a ball can be approximated, after this
rescaling, by suitable approximate Laplace eigenfunctions;
see \cite[Appendix~A]{10.1093/imrn/rnac282}. Inverse localization asks whether this approximation can instead be realized by genuine eigenfunctions.

In the currently known closed-manifold examples, this is achieved by exploiting large spectral multiplicities, which provide enough freedom inside an eigenspace to reproduce a prescribed local Helmholtz solution. The underlying localization idea first appeared in \cite{Paco} for Beltrami fields on the three-dimensional sphere and flat torus. For Laplace eigenfunctions, inverse localization is known to hold for certain flat tori \cite{10.1093/imrn/rnac282} (while it is known to fail for most of them), the round sphere $\mathbb{S}^d$ and its Riemannian quotients \cite{Enciso2021}; a Dirichlet version for certain Euclidean polygons is studied in \cite{Polygons}. We first recall the form of the property used in these works.

\begin{definition}[Inverse localization]\label{IL}
	A compact manifold $M$ has the inverse localization property if,  for every framed partition $(U_m,e^m)_{m=1}^{m_{\max}}$
	as in Lemma~\ref{frames}, every $r\in\mathbb N_0$, every ball $B\subset\mathbb{R}^d$ centered at the origin, and every solution $h:\mathbb{R}^d\to\mathbb{R}$ of the Helmholtz equation, there exist eigenvalues
	\[
	\lambda_{n_k}\xrightarrow[k\to\infty]{}\infty
	\]
	and associated eigenfunctions $\psi_{n_k}$ satisfying
	\[
	\Delta_g\psi_{n_k}+\lambda_{n_k}\psi_{n_k}=0,
	\]
	together with an index $m$ and points $p_k\in U_m$, such that
\begin{equation}
	\left\|
	\psi_{n_k}\left(
	\exp_{p_k}\left(
	\frac{e_{p_k}^m(\cdot)}{\sqrt{\lambda_{n_k}}}
	\right)
	\right)-h
	\right\|_{C^r(\overline B)}
	\xrightarrow[k\to\infty]{}0.
\end{equation}
The eigenvalues, eigenfunctions, index $m$, and points may
depend on the partition, $r,B,$ and $h$.
\end{definition}

In the known boundaryless examples, the localization point can in fact be fixed in advance, whereas for polygons it may depend on $h$ and $\ep$. More generally, formulations of inverse localization may differ in the approximation norm, the domain, or the allowed dependence of the localization point.

The Berry property considered here yields the following stronger form.

\begin{definition}[Strong inverse localization]\label{SIL}
	A compact manifold $M$ has the strong inverse localization property if it admits a real Laplace eigenbasis $(\psi_n)_{n\in\mathbb{N}_0}$ such that, for every framed partition $(U_m,e^m)_{m=1}^{m_{\max}}$
	as in Lemma~\ref{frames}, for every $r\in\mathbb{N}_0$, every ball $B\subset\mathbb{R}^d$, every $h\in\mathrm{FP}$, every $\ep>0$, and every $1\leq m\leq m_{\max}$, there exists $n_0\in\mathbb{N}$ such that, for every $n\geq n_0$, the set
	\begin{equation}
		N_m(r,B,n,h,\ep)
		:=
		\left\{
		p\in U_m:
		\|\psi_{p,m}^n-h\|_{C^r(\overline B)}<\ep
		\right\}
	\end{equation}
	has positive Riemannian volume.
\end{definition}

The difference between the two notions is twofold. In Definition~\ref{IL}, the eigenvalues and eigenfunctions used in the approximation may depend on the target $h$, and it is enough to realize the approximation along a suitable subsequence. In Definition~\ref{SIL}, by contrast, the eigenbasis is fixed independently of $h$, and for every sufficiently large $n$ there is a positive-volume set of localization points where the approximation holds. Thus the freedom lies in the choice of the localization point rather than in the choice of an eigenfunction within an eigenspace. In particular, strong inverse localization implies inverse localization.

We now prove that the Berry property implies this stronger form of inverse localization.

\begin{theorem}
	If $(M,g)$ is a compact connected smooth Riemannian manifold such that a normalized real Laplace eigenbasis
	satisfies the Berry property in any of the formulations considered here, then it satisfies strong inverse localization.
\end{theorem}

\begin{proof}
	Let $(\psi_n)_{n\in\mathbb{N}_0}$ be the real eigenbasis satisfying Berry's property as in Definition~\ref{C1}, normalized by
	\[
	\|\psi_n\|_{L^2(M,dV_g)}^2=\Vol_g(M).
	\]
	Fix $\ep>0$, $r\in\mathbb{N}\cup\{0\}$, a ball $B\subset\mathbb{R}^d$, $h\in\mathrm{FP}$, and $1\leq m\leq m_{\max}$. Choose a ball $B'$ with $\overline B\subset B'$, and let
	\[
	\chi\in C^\infty([0,\infty);[0,1])
	\]
	satisfy
	\[
	\chi(t)=1\quad\text{for }0\leq t\leq\ep/2,
	\qquad
	\chi(t)=0\quad\text{for }t\geq\ep.
	\]
	Define the bounded continuous functional
	\[
	F(f):=
	\chi\bigl(\|f-h\|_{C^r(\overline B)}\bigr),
	\qquad
	f\in C^\infty(\mathbb{R}^d).
	\]
	Thus $F(f)=1$ whenever $\|f-h\|_{C^r(\overline B)}\leq\ep/2$, while, on the other hand, $F(f)=0$ whenever $\|f-h\|_{C^r(\overline B)}\geq\ep$.	In particular, for any $n$ and $p$, we have \begin{equation}
		F\left(\psi_{p,m}^n\right) =\chi\left(\left\|\psi_{p,m}^n-h\right\|_{C^r(\overline B)}\right).
	\end{equation}
	By the Berry property, we get
	\begin{equation}\label{eq:berry-cutoff}
	\frac{1}{\Vol_g(U_m)}
		\int_{U_m}
		F(\psi_{p,m}^n)\,dV_g(p)
		\xrightarrow[n\to\infty]{}
		\mathbb{E}\bigl[F(\Psi_{\mathrm{RWM}})\bigr]
		=:c_{h,B,r,\ep}.
	\end{equation}
	
	We claim that $c_{h,B,r,\ep}>0$. By the interior elliptic estimates, there exists a constant $C=C(B,B',r)>0$ such that
	\begin{equation}\label{eq:elliptic-Cr}
		\|u\|_{C^r(\overline B)}
		\leq
		C\|u\|_{C^0(\overline{B'})}
	\end{equation}
	for every $u\in\mathrm{FP}$. Hence it is enough to approximate $h$ uniformly on $B'$.
	
	Let $\mu_{B'}$ be the law of the restriction of
	$\Psi_{\mathrm{RWM}}$ to $\overline{B'}$.
	By the Gaussian support theorem~\cite[Lemma~A.4]{NZ2016},
	its support in $C^0(\overline{B'})$ is the closure of
	the reproducing-kernel Hilbert space $\mathcal H_{B'}$
	of the restricted field. Thus it suffices to approximate
	$h$ uniformly on $\overline{B'}$ by elements of $\mathcal H_{B'}$.
	
	The spectral representation of the covariance in \eqref{covRWM}
	shows that $\mathcal H_{B'}$ contains the restrictions
	of the Herglotz waves
	\begin{equation}
		H_a(x):=
		\int_{\mathbb S^{d-1}}
		e^{ix\cdot\xi}a(\xi)\,d\omega_{d-1}(\xi),
		\qquad
		a\in L^2(\mathbb S^{d-1},d\omega_{d-1}),
		\quad
		a(-\xi)=\overline{a(\xi)},
	\end{equation}
	where the last condition makes $H_a$ real-valued.
	By \cite[Theorem~3.1]{10.1093/imrn/rnac282}, applied on
	a ball containing $\overline{B'}$, $h$ can be approximated
	uniformly on $\overline{B'}$ by such waves with polynomial
	densities. Their restrictions belong to $\mathcal H_{B'}$,
	so
	\[
	h|_{\overline{B'}}\in
	\overline{\mathcal H_{B'}}^{\,C^0(\overline{B'})}
	=\supp\mu_{B'}.
	\]
	It follows that
	\begin{equation}\label{eq:C0-positive-mass}
		\mu_{\mathrm{RWM}}
		\left(
		\left\{
		f:
		\|f-h\|_{C^0(\overline{B'})}<\frac{\ep}{2C}
		\right\}
		\right)>0.
	\end{equation}
	
	Since both $f$ and $h$ solve the Helmholtz equation for $\mu_{\mathrm{RWM}}$-almost every $f$, applying \eqref{eq:elliptic-Cr} to $f-h$ shows that, up to a null set, the event in \eqref{eq:C0-positive-mass} is contained in
	\[
	\left\{
	f:
	\|f-h\|_{C^r(\overline B)}<\ep/2
	\right\}.
	\]
	On this set $F(f)=1$, so $c_{h,B,r,\ep}>0$.
	
	Finally,
	\[
	0\leq F
	\leq
	\mathbf{1}_{\{f:\|f-h\|_{C^r(\overline B)}<\ep\}},
	\]
	and therefore \eqref{eq:berry-cutoff} gives
	\begin{equation}
		\liminf_{n\to\infty}
		\frac{\Vol_g\bigl(N_m(r,B,n,h,\ep)\bigr)}
		{\Vol_g(U_m)}
		\geq
		c_{h,B,r,\ep}>0.
	\end{equation}
	Hence $N_m(r,B,n,h,\ep)$ has positive Riemannian volume for every sufficiently large $n$, which is precisely the strong inverse localization property.
\end{proof}

We conclude with several consequences and remarks.
First, for fixed $r,B,h$ and $m$, the theorem yields points
$p(n)\in U_m$ such that
\[
\|\psi_{p(n),m}^n-h\|_{C^r(\overline B)}\longrightarrow0.
\]
Indeed, the infimum of this error over $p\in U_m$ tends
to zero by the theorem, and we choose $p(n)$ with error
less than this infimum plus $1/n$. 

Second, for fixed $r,B,h,\ep$ and $m$, the admissible
localization points are asymptotically dense in $U_m$.
Indeed, Lemma~\ref{subset} gives the local weak limit
property on every nonempty open set $U\subset U_m$.
Applying the preceding proof with the point restricted
to $U$ shows that every such $U$ contains admissible localization points for all sufficiently large $n$, and
\[
\Vol_g\bigl(U\cap N_m(r,B,n,h,\ep)\bigr)>0
\]
for all sufficiently large $n$.

A direct consequence of the theorem is the following obstruction.

\begin{corollary}
	If $(M,g)$ does not satisfy strong inverse localization, then no normalized real Laplace eigenbasis on $M$ can satisfy
	the Berry property. In particular, the irrational flat tori shown in \cite{10.1093/imrn/rnac282} not to satisfy inverse localization cannot satisfy Berry's property.
\end{corollary}

Inverse localization also yields local information on nodal sets. The following result is stated in \cite{10.1093/imrn/rnac282} for arithmetic tori and extends to manifolds satisfying inverse localization. Before introducing it, we recall the condition imposed
on the relative positions of the prescribed hypersurfaces.

\begin{definition}
	A collection of compact connected smooth embedded hypersurfaces without boundary
	$\Sigma_j\subset\mathbb{R}^d$, $1\leq j\leq N$, is said to be
	\emph{not linked} if there exist pairwise disjoint contractible
	domains $V_j\subset\mathbb{R}^d$ with $\Sigma_j\subset V_j$.
\end{definition}

This condition allows the associated domains to be chosen
disjoint and rescaled to have the same first Dirichlet eigenvalue.

\begin{theorem}\label{T.nodalsets}
	If a manifold $M$ satisfies either Definition~\ref{IL} or Definition~\ref{SIL}, then, given any $N\in\mathbb{N}$, any collection of compact embedded hypersurfaces $\Sigma_j\subset\mathbb{R}^d$, $1\leq j\leq N$, that are not linked, any positive integer $r$, and any $\ep>0$, there exists $R>0$ such that, for arbitrarily large $n$,	there are an eigenfunction $\psi_n$ with eigenvalue $\lambda_n$,
	an index $m$, and a point $p\in U_m$ such that
	\begin{equation}
		\psi_n\left(
		\exp_p\left(\frac{e_p^m(\cdot)}{\sqrt{\lambda_n}}\right)
		\right)
	\end{equation}
	has, in the ball of radius $R$, at least $N$ nodal components of the form
	\begin{equation}
		\widetilde\Sigma_j^n
		=
		\Phi_n(c_j\Sigma_j+p_j),
	\end{equation}
	and at least $N$ nondegenerate local extrema. Here $c_j>0$ and $p_j\in\mathbb{R}^d$ depend on the chosen collection and are chosen so that the hypersurfaces $(c_j\Sigma_j+p_j)_{1\leq j\leq N}$ are disjoint, while $\Phi_n$ is a diffeomorphism of $\mathbb{R}^d$ satisfying
	\[
	\|\Phi_n-\id\|_{C^r(\mathbb{R}^d)}<\ep
	\]
	and may depend on both $\psi_n$ and the collection $(\Sigma_j)_{1\leq j\leq N}$.
\end{theorem}
\begin{proof}
	The Euclidean construction used in
	\cite{10.1093/imrn/rnac282} gives $R>0$ and
	$h\in\mathrm{FP}$ with $N$ structurally stable nodal
	components of the prescribed form and $N$ nondegenerate
	local extrema in $B_R$. Choosing the initial deformation
	sufficiently small, these conclusions, including the stated
	$C^r$ bound, persist under sufficiently small
	$C^{r+2}(\overline{B_R})$ perturbations of $h$.
	Definition~\ref{IL} provides such rescaled eigenfunctions.
	Under Definition~\ref{SIL}, they may be taken from the fixed
	basis for sufficiently large $n$.
\end{proof}
Therefore, Berry's property implies, through strong inverse localization, that prescribed local nodal configurations can be realized in suitable rescalings of the eigenfunctions in the fixed sequence.

Finally, Definition~\ref{C1} also makes sense for manifolds with boundary and Dirichlet eigenfunctions, as in \cite{Ingremeau2021}; inverse localization in that setting is studied in \cite{Polygons}.

\section*{Acknowledgments}
I am deeply grateful to Maxime Ingremeau for his essential support
and help with technical aspects, to Daniel Peralta-Salas and Alberto Enciso for their useful comments, and to the anonymous referee for a careful reading, corrections, and suggestions.

This work was supported by the ERC under Horizon 2020
(grant~862342), a postgraduate fellowship from the Ministry of
Science and Innovation at the Residencia de Estudiantes
(2022--2024), and the Department of Mathematics of CUNEF Universidad.

I used AI-assisted tools to help simplify the proofs
of Lemmas~\ref{equiv} and~\ref{lem:framed-pullback}.
I take full responsibility for the final text and mathematical arguments.
\section*{Declarations}
\subsection*{Conflict of interest}
The author declares no conflict of interest.
\bibliographystyle{amsplain}
\raggedbottom
\bibliography{export}

@misc{ABLM,
  author       = {Ab{\'e}rt, M. and Bergeron, N. and Le Masson, E.},
  title        = {Eigenfunctions and random waves in the {Benjamini--Schramm} limit},
  year         = {2018},
  note         = {Preprint, arXiv:1810.05601 (version~2, 2022)},
  doi          = {10.48550/arXiv.1810.05601},
  eprint       = {1810.05601},
  archivePrefix= {arXiv},
  primaryClass = {math.SP},
  url          = {https://arxiv.org/abs/1810.05601},
}

@article{Ingremeau2021,
  author       = {Ingremeau, M.},
  title        = {Local weak limits of {Laplace} eigenfunctions},
  journal      = {Tunis. J. Math.},
  volume       = {3},
  number       = {3},
  pages        = {481--515},
  year         = {2021},
  doi          = {10.2140/tunis.2021.3.481},
  url          = {https://doi.org/10.2140/tunis.2021.3.481},
}

@article{NZ2016,
  author       = {Nazarov, F. and Sodin, M.},
  title        = {Asymptotic laws for the spatial distribution and the number of connected components of zero sets of {Gaussian} random functions},
  journal      = {J. Math. Phys. Anal. Geom.},
  volume       = {12},
  number       = {3},
  pages        = {205--278},
  year         = {2016},
  doi          = {10.15407/mag12.03.205},
  url          = {https://doi.org/10.15407/mag12.03.205},
}

@article{BS,
  author       = {Ab{\'e}rt, M. and Biringer, I.},
  title        = {Unimodular measures on the space of all {Riemannian} manifolds},
  journal      = {Geom. Topol.},
  volume       = {26},
  number       = {5},
  pages        = {2295--2404},
  year         = {2022},
  doi          = {10.2140/gt.2022.26.2295},
  url          = {https://doi.org/10.2140/gt.2022.26.2295},
}

@book{cohen2013fractional,
  author       = {Cohen, S. and Istas, J.},
  title        = {Fractional Fields and Applications},
  series       = {Math{\'e}matiques et Applications},
  volume       = {73},
  publisher    = {Springer},
  address      = {Heidelberg},
  year         = {2013},
  doi          = {10.1007/978-3-642-36739-7},
  url          = {https://doi.org/10.1007/978-3-642-36739-7},
}

@article{Berry1977RegularAI,
  author       = {Berry, M. V.},
  title        = {Regular and irregular semiclassical wavefunctions},
  journal      = {J. Phys. A: Math. Gen.},
  volume       = {10},
  number       = {12},
  pages        = {2083--2091},
  year         = {1977},
  doi          = {10.1088/0305-4470/10/12/016},
  url          = {https://doi.org/10.1088/0305-4470/10/12/016},
}

@incollection{Berry1983,
  author       = {Berry, M. V.},
  title        = {Semiclassical mechanics of regular and irregular motion},
  booktitle    = {Chaotic Behaviour of Deterministic Systems},
  editor       = {Iooss, G. and Helleman, R. H. G. and Stora, R.},
  pages        = {171--271},
  publisher    = {North-Holland},
  address      = {Amsterdam},
  year         = {1983},
}

@article{AURICH1993185,
  author       = {Aurich, R. and Steiner, F.},
  title        = {Statistical properties of highly excited quantum eigenstates of a strongly chaotic system},
  journal      = {Phys. D},
  volume       = {64},
  number       = {1--3},
  pages        = {185--214},
  year         = {1993},
  doi          = {10.1016/0167-2789(93)90255-Y},
  url          = {https://doi.org/10.1016/0167-2789(93)90255-Y},
}

@article{PhysRevE.57.5425,
  author       = {B{\"a}cker, A. and Schubert, R. and Stifter, P.},
  title        = {Rate of quantum ergodicity in {Euclidean} billiards},
  journal      = {Phys. Rev. E},
  volume       = {57},
  number       = {5},
  pages        = {5425--5447},
  year         = {1998},
  doi          = {10.1103/PhysRevE.57.5425},
  url          = {https://doi.org/10.1103/PhysRevE.57.5425},
}

@article{Barnett2005AsymptoticRO,
  author       = {Barnett, A. H.},
  title        = {Asymptotic rate of quantum ergodicity in chaotic {Euclidean} billiards},
  journal      = {Comm. Pure Appl. Math.},
  volume       = {59},
  number       = {10},
  pages        = {1457--1488},
  year         = {2006},
  doi          = {10.1002/cpa.20150},
  url          = {https://doi.org/10.1002/cpa.20150},
}

@article{PhysRevE.70.056209,
  author       = {Savytskyy, N. and Hul, O. and Sirko, L.},
  title        = {Experimental investigation of nodal domains in the chaotic microwave rough billiard},
  journal      = {Phys. Rev. E},
  volume       = {70},
  number       = {5},
  pages        = {056209},
  year         = {2004},
  doi          = {10.1103/PhysRevE.70.056209},
  url          = {https://doi.org/10.1103/PhysRevE.70.056209},
}

@article{PhysRevE.72.066212,
  author       = {Hul, O. and Savytskyy, N. and Tymoshchuk, O. and Bauch, S. and Sirko, L.},
  title        = {Investigation of nodal domains in the chaotic microwave ray-splitting rough billiard},
  journal      = {Phys. Rev. E},
  volume       = {72},
  number       = {6},
  pages        = {066212},
  year         = {2005},
  doi          = {10.1103/PhysRevE.72.066212},
  url          = {https://doi.org/10.1103/PhysRevE.72.066212},
}

@article{HohStock,
  author       = {Kuhl, U. and H{\"o}hmann, R. and St{\"o}ckmann, H.-J. and Gnutzmann, S.},
  title        = {Nodal domains in open microwave systems},
  journal      = {Phys. Rev. E},
  volume       = {75},
  number       = {3},
  pages        = {036204},
  year         = {2007},
  doi          = {10.1103/PhysRevE.75.036204},
  url          = {https://doi.org/10.1103/PhysRevE.75.036204},
}

@article{benjamini2001recurrence,
  author       = {Benjamini, I. and Schramm, O.},
  title        = {Recurrence of distributional limits of finite planar graphs},
  journal      = {Electron. J. Probab.},
  volume       = {6},
  number       = {23},
  pages        = {1--13},
  year         = {2001},
  doi          = {10.1214/EJP.v6-96},
  url          = {https://doi.org/10.1214/EJP.v6-96},
}

@incollection{zelditch2010recent,
  author       = {Zelditch, S.},
  title        = {Recent developments in mathematical quantum chaos},
  booktitle    = {Current Developments in Mathematics, 2009},
  pages        = {115--204},
  publisher    = {International Press},
  address      = {Somerville, MA},
  year         = {2010},
  doi          = {10.4310/CDM.2009.v2009.n1.a4},
  url          = {https://doi.org/10.4310/CDM.2009.v2009.n1.a4},
}

@incollection{Nonnenmacher2013,
  author       = {Nonnenmacher, S.},
  title        = {Anatomy of quantum chaotic eigenstates},
  booktitle    = {Chaos: Poincar{\'e} Seminar 2010},
  editor       = {Duplantier, B. and Nonnenmacher, S. and Rivasseau, V.},
  series       = {Progress in Mathematical Physics},
  volume       = {66},
  pages        = {193--238},
  publisher    = {Birkh{\"a}user},
  address      = {Basel},
  year         = {2013},
  doi          = {10.1007/978-3-0348-0697-8_6},
  url          = {https://doi.org/10.1007/978-3-0348-0697-8_6},
}

@article{Hejhal1992,
  author       = {Hejhal, D. A. and Rackner, B. N.},
  title        = {On the topography of {Maass} waveforms for {$\mathrm{PSL}(2,\mathbb{Z})$}},
  journal      = {Experiment. Math.},
  volume       = {1},
  number       = {4},
  pages        = {275--305},
  year         = {1992},
  doi          = {10.1080/10586458.1992.10504562},
  url          = {https://doi.org/10.1080/10586458.1992.10504562},
}

@article{Humphries2017EquidistributionIS,
  author       = {Humphries, P.},
  title        = {Equidistribution in shrinking sets and {$L^4$}-norm bounds for automorphic forms},
  journal      = {Math. Ann.},
  volume       = {371},
  number       = {3--4},
  pages        = {1497--1543},
  year         = {2018},
  doi          = {10.1007/s00208-018-1677-9},
  url          = {https://doi.org/10.1007/s00208-018-1677-9},
}

@incollection{Enciso2021,
  author       = {Enciso, A. and Peralta-Salas, D. and Torres de Lizaur, F.},
  title        = {High-energy eigenfunctions of the {Laplacian} on the torus and the sphere with nodal sets of complicated topology},
  booktitle    = {Nonlinear Partial Differential Equations for Future Applications},
  editor       = {Koike, S. and Kozono, H. and Ogawa, T. and Sakaguchi, S.},
  series       = {Springer Proceedings in Mathematics \& Statistics},
  volume       = {346},
  pages        = {245--261},
  publisher    = {Springer},
  address      = {Singapore},
  year         = {2021},
  doi          = {10.1007/978-981-33-4822-6_7},
  url          = {https://doi.org/10.1007/978-981-33-4822-6_7},
}

@article{10.1093/imrn/rnac282,
  author       = {Enciso, A. and Garc{\'i}a-Ruiz, A. and Peralta-Salas, D.},
  title        = {Localization properties of high-energy eigenfunctions on flat tori},
  journal      = {Int. Math. Res. Not. IMRN},
  volume       = {2023},
  number       = {24},
  pages        = {20988--21014},
  year         = {2023},
  doi          = {10.1093/imrn/rnac282},
  url          = {https://doi.org/10.1093/imrn/rnac282},
}

@article{hormander1969spectral,
  author       = {H{\"o}rmander, L.},
  title        = {The spectral function of an elliptic operator},
  journal      = {Acta Math.},
  volume       = {121},
  pages        = {193--218},
  year         = {1968},
  doi          = {10.1007/BF02391913},
  url          = {https://doi.org/10.1007/BF02391913},
}

@book{einsiedler2010ergodic,
  author       = {Einsiedler, M. and Ward, T.},
  title        = {Ergodic Theory: With a View Towards Number Theory},
  series       = {Graduate Texts in Mathematics},
  volume       = {259},
  publisher    = {Springer-Verlag},
  address      = {London},
  year         = {2011},
  doi          = {10.1007/978-0-85729-021-2},
  url          = {https://doi.org/10.1007/978-0-85729-021-2},
}

@article{graph,
  author       = {Backhausz, {\'A}. and Szegedy, B.},
  title        = {On the almost eigenvectors of random regular graphs},
  journal      = {Ann. Probab.},
  volume       = {47},
  number       = {3},
  pages        = {1677--1725},
  year         = {2019},
  doi          = {10.1214/18-AOP1294},
  url          = {https://doi.org/10.1214/18-AOP1294},
}

@article{Bourgain,
  author       = {Bourgain, J.},
  title        = {On toral eigenfunctions and the random wave model},
  journal      = {Israel J. Math.},
  volume       = {201},
  number       = {2},
  pages        = {611--630},
  year         = {2014},
  doi          = {10.1007/s11856-014-1037-z},
  url          = {https://doi.org/10.1007/s11856-014-1037-z},
}

@article{ROMANIEGA20221,
  author       = {Romaniega, {\'A}. and Sartori, A.},
  title        = {Nodal set of monochromatic waves satisfying the {Random Wave Model}},
  journal      = {J. Differential Equations},
  volume       = {333},
  pages        = {1--54},
  year         = {2022},
  doi          = {10.1016/j.jde.2022.05.023},
  url          = {https://doi.org/10.1016/j.jde.2022.05.023},
}

@article{Paco,
  author       = {Enciso, A. and Peralta-Salas, D. and Torres de Lizaur, F.},
  title        = {Knotted structures in high-energy {Beltrami} fields on the torus and the sphere},
  journal      = {Ann. Sci. {\'E}c. Norm. Sup{\'e}r. (4)},
  volume       = {50},
  number       = {4},
  pages        = {995--1016},
  year         = {2017},
  doi          = {10.24033/asens.2337},
  url          = {https://doi.org/10.24033/asens.2337},
}

@misc{ingremeau_lecture_notes,
  author       = {Ingremeau, M.},
  title        = {Random limits of highly oscillatory deterministic functions},
  year         = {2025},
  note         = {Lecture notes for the Dynqua meeting in Angers, February 2025},
  url          = {https://www-fourier.univ-grenoble-alpes.fr/~ingremem/Lecture%20Notes%20Angers.pdf},
}

@book{Billingsley1999,
  author       = {Billingsley, P.},
  title        = {Convergence of Probability Measures},
  edition      = {2nd},
  series       = {Wiley Series in Probability and Statistics},
  publisher    = {John Wiley \& Sons},
  address      = {New York},
  year         = {1999},
  doi          = {10.1002/9780470316962},
  url          = {https://doi.org/10.1002/9780470316962},
}

@misc{Polygons,
  author       = {Enciso, A. and Garc{\'i}a-Ruiz, A. and Peralta-Salas, D.},
  title        = {Local limits of high energy eigenfunctions on integrable billiards},
  year         = {2025},
  note         = {Preprint, arXiv:2502.01291},
  doi          = {10.48550/arXiv.2502.01291},
  eprint       = {2502.01291},
  archivePrefix= {arXiv},
  url          = {https://arxiv.org/abs/2502.01291},
}

@incollection{AldousSteele2004,
  author       = {Aldous, D. and Steele, J. M.},
  title        = {The objective method: Probabilistic combinatorial optimization and local weak convergence},
  booktitle    = {Probability on Discrete Structures},
  editor       = {Kesten, H.},
  series       = {Encyclopaedia of Mathematical Sciences},
  volume       = {110},
  pages        = {1--72},
  publisher    = {Springer},
  address      = {Berlin},
  year         = {2004},
  doi          = {10.1007/978-3-662-09444-0_1},
  url          = {https://doi.org/10.1007/978-3-662-09444-0_1},
}

@article{AldousLyons2007,
  author       = {Aldous, D. and Lyons, R.},
  title        = {Processes on unimodular random networks},
  journal      = {Electron. J. Probab.},
  volume       = {12},
  number       = {54},
  pages        = {1454--1508},
  year         = {2007},
  doi          = {10.1214/EJP.v12-463},
  url          = {https://doi.org/10.1214/EJP.v12-463},
}

@misc{LewisZhang2023,
  author       = {Lewis, A. D. and Zhang, Y.},
  title        = {Topologies for geometric flows and continuous dependence on parameters},
  year         = {2023},
  note         = {Preprint, arXiv:2310.12293},
  doi          = {10.48550/arXiv.2310.12293},
  eprint       = {2310.12293},
  archivePrefix= {arXiv},
  primaryClass = {math.GT},
  url          = {https://arxiv.org/abs/2310.12293},
}

@book{Treves1967,
  author       = {Tr{\`e}ves, F.},
  title        = {Topological Vector Spaces, Distributions and Kernels},
  series       = {Pure and Applied Mathematics},
  volume       = {25},
  publisher    = {Academic Press},
  address      = {New York},
  year         = {1967},
}

@book{EthierKurtz1986,
  author       = {Ethier, S. N. and Kurtz, T. G.},
  title        = {Markov Processes: Characterization and Convergence},
  series       = {Wiley Series in Probability and Mathematical Statistics},
  publisher    = {John Wiley \& Sons},
  address      = {New York},
  year         = {1986},
  doi          = {10.1002/9780470316658},
}

@article{Ingremeau2022,
  author       = {Ingremeau, M. and Rivera, A.},
  title        = {How {Lagrangian} states evolve into random waves},
  journal      = {J. {\'E}c. polytech. Math.},
  volume       = {9},
  pages        = {177--212},
  year         = {2022},
  doi          = {10.5802/jep.181},
  url          = {https://doi.org/10.5802/jep.181},
}
\end{document}